\documentclass[12pt,reqno]{amsart}
\usepackage{geometry}                		
\usepackage{graphicx}				
				
\usepackage{amssymb}
\usepackage{epstopdf}
\usepackage{datetime}
\usepackage{currfile} 

\usepackage{color}

\usepackage{bm} 

\newtheorem{thm}{Theorem}[section]
\newtheorem{alg}{Algorithm}[section]

\newtheorem{lem}[thm]{Lemma}

\theoremstyle{definition}
\newtheorem{defn}{Definition}[section]
\theoremstyle{remark}
\newtheorem{rem}{Remark}[section]

\numberwithin{equation}{section}
						
\numberwithin{equation}{section}

\newcommand{\red}[1]{\textcolor{black}{#1}}

\begin{document}

\title[IRP DG methods for hyperbolic conservation law systems]
      {Invariant-region-preserving DG methods for multi-dimensional hyperbolic conservation law systems, with an application to compressible Euler equations}
       
       \author{
Yi Jiang and Hailiang Liu
}
\address{Iowa State University, Mathematics Department, Ames, IA 50011} \email{yjiang1@iastate.edu}
\address{Iowa State University, Mathematics Department, Ames, IA 50011} \email{hliu@iastate.edu}

\date{November 1, 2017; revised Feb 28, 2018} 
\thanks{This work was supported by the National Science Foundation under Grant DMS1312636 and by NSF Grant RNMS (Ki-Net) 1107291.
}
\subjclass[2000]{65M60, 35L65, 35L45}

\keywords{Invariant region, hyperbolic conservation laws,  compressible Euler equation,  discontinuous Galerkin methods}


\begin{abstract} 
An invariant-region-preserving (IRP) limiter for multi-dimensional hyperbolic conservation law systems is introduced, as long as the system admits a global invariant region which is a convex set in the phase space. It is shown that the order of approximation accuracy is not destroyed by the IRP limiter, 
\red{provided the cell average is away from the boundary of the convex set}. 
 Moreover, this limiter is explicit, and easy for computer implementation.
 A generic algorithm incorporating the IRP limiter is presented for high order finite volume type schemes. 
 For arbitrarily high order discontinuous Galerkin (DG) schemes to hyperbolic conservation law systems, sufficient conditions are obtained for cell averages to remain in the invariant region provided the projected one-dimensional system shares the same invariant region as the full multi-dimensional hyperbolic system {does}. The general results are then applied to both one and two dimensional compressible Euler equations so to obtain high order IRP DG schemes. Numerical experiments are provided to validate the proven properties of the IRP limiter and the performance of IRP DG schemes for compressible Euler equations.
\end{abstract}

\maketitle
 \section{Introduction}
The multi-dimensional hyperbolic conservation law systems are given by 
  \begin{equation}\label{mcl}
\partial_t \mathbf{w}+\sum_{j=1}^d \partial_{x_j} F_j(\mathbf{w}) =0,  \quad x\in \mathbb{R}^d, \quad t>0
\end{equation}
with the unknown vector $\mathbf{w}\in \mathbb{R}^l$ and the flux function $F_j(\mathbf{w}) \in \mathbb{R}^l$ for $j=1, \cdots d$. We consider the initial value problem for system (\ref{mcl}) with the initial data $ \mathbf{w}(x, 0)=\mathbf{w}_0(x)$. For simplicity, we take periodic or compactly supported boundary conditions.

It is well known that entropy inequalities should be considered for general hyperbolic conservation laws so to single out the physically relevant solution among many weak solutions (see, e.g., \cite{La73}). In application problems, the pointwise range of solutions may be known from physical considerations, instead of  total entropy.  For scalar conservation laws, the entropy solution satisfies a strict maximum principle.  For hyperbolic conservation law systems, the notion of maximum principle does not apply and must be replaced by the notion of invariant region. To solve a conservation law system with possibly discontinuous solutions, one naturally studies the invariant-region-preserving (IRP) property of the numerical schemes.

 In this paper,  we are interested in constructing IRP high order accurate schemes for solving (\ref{mcl}) with an application to compressible Euler equations. The invariant region $\Sigma$ to this system is meant to be a convex set in phase space $\mathbb{R}^l$ so that if the initial data is in the region $\Sigma$, then the solution will remain in $\Sigma$ for all $t>0$. It is highly desirable to construct high order numerical schemes solving (\ref{mcl}) that can preserve the entire invariant region $\Sigma$, which is in general a difficult problem.  In this article, we will discuss the IRP property of arbitrarily high order schemes on shape-regular meshes, following the discontinuous Galerkin (DG) framework developed by Cockburn and Shu \cite{CLS89, CS89, CS98}.   

There are models that feature known invariant regions. For example, the invariant region of  one-dimensional  $2\times 2$ ($l=2$) systems of hyperbolic conservation laws can be described by two Riemann invariants, see e.g.,  \cite{CCS77, Hoff85, Sm94, Fr01}. For general conservation law systems with $l\geq 3$, it is an open question to identify a global invariant region.  When considering the compressible Euler equations of gas dynamics, a natural  condition for the solution is positivity of density and pressure, and the {minimum} principle for the specific entropy \cite{Ta86}.
 In the one dimensional case, the Euler equation has the following form 
  \begin{align}\label{euler}
& \partial_t \mathbf{w}+ \partial_x f(\mathbf{w})=0, \quad t>0, \; x\in \mathbb{R}, 
\end{align}
with $\mathbf{w}=(\rho, m, E)^\top$,  
\begin{align}\label{eq:defenergy}
 f(\mathbf{w})=(m, \rho u^2+ p, (E+p)u)^\top,\;  E=\frac{1}{2}\rho u^2 +\frac{p}{\gamma-1},
\end{align}
 where $\rho$  is the density, $u$ is the velocity, $m=\rho u$ is the momentum, $E$ is the total energy,  $p$ is the pressure, 
 \red{and  $\gamma$ is the ratio of specific heats for the gas/fluid (for most gases, $1<\gamma < 3)$}. The corresponding invariant region is  the following set  
$$
G=\left\{\mathbf{w}\big | \quad   \rho>0, \quad  p>0,  
\quad s \geq  s_0 \right\},
$$ 
where $s=\log (p/\rho^\gamma)$ is the specific entropy, and $s_0:=\inf \limits_x s(\mathbf{w}_0(x))$. 

For nonlinear systems of conservation laws in several space variables with a known invariant region, the IRP property under the Lax-Friedrich schemes was studied by Frid in \cite{Fr95,Fr01}. For the compressible Euler equations, the first order finite volume schemes including Godunov and Lax-Friedrichs schemes are shown to preserve the minimum entropy property \cite{Ta86}. Further second-order limitation techniques  for enforcing the specific entropy bound were proposed in \cite{KP94}, where it was reported that enforcing the minimum entropy principle numerically might damp oscillations in numerical solutions.  
To have the specific entropy well-defined, one would have to guarantee the positivity of density and pressure of the numerical solution, which can be done for a high order finite volume or a DG scheme following \cite{PQ94, PS96, ZS10b, ZXS12}. The main idea of positivity-preserving
techniques for high order DG schemes in \cite{ZS10b, ZXS12} is to find a sufficient condition to preserve the positivity of the cell averages by repeated convex combinations, combined with a conservative limiter which can enforce the sufficient condition without destroying accuracy for 
smooth solutions, as  shown in \cite{ZS10a} for scalar conservation laws. 
In the context of continuous finite elements, the IRP property has been studied by Guermond and Popov \cite{GP16} using the first order approximation to solve general hyperbolic conservation law systems, and then in \cite{GP17} using the second order approximation with convex limiting to solve compressible Euler equations.

A more closely related development is the work by Zhang and Shu \cite{ZS12}, where the authors extended the positivity-preserving high order schemes for compressible Euler equations to preserve the entire $G$, while their limiter for enforcing the lower bound of $s$ is implicit with the limiter parameter solved by the Newton iteration. 
In \cite{Hyp16}, we introduced an explicit  limiter based on  a simple observation that  the convex set $G$ can be rewritten as
 \begin{equation}\label{s1}
\Sigma= \left\{ \mathbf{w}\big | \quad   \rho>0, \quad  p>0,  
\quad q \leq 0 
 \right\},  
 \end{equation}
where $q=(s_0-s)\rho$. Note that $s$ is quasi-concave, but $q$ is convex; actually the fact that $-\rho s$ has a positive definite Hessian matrix can be derived from a general result by Harten in \cite{H83}. Such a reformulation allows us to construct a new IRP limiter in \cite{Hyp16} for compressible Euler equations. The limiter modifies the polynomial solution still through a linear convex combination as  in \cite{ZS12}, 
yet the limiter parameter is defined explicitly due to the convexity of $q$, and concavity of $p$. \red{The question of particular interest is whether it is still high order accurate.   In \cite{Hyp16},  the IRP limiter was proved  to maintain same high order accuracy if the cell average is away from the boundary of the convex set. While  the bound preserving limiter \cite{ZS12} for the entropy function was shown to be high order accurate provided the second order derivative of the entropy function for numerical solutions does not vanish.}

The work \cite{Hyp16} was built upon  \cite{JLPsy},  where we introduced an explicit IRP limiter for DG methods solving the isentropic gas dynamic system (with or without viscosity). Again both convexity and concavity of two Riemann invariants play an essential role in the construction of the explicit IRP limiter {there}.
We observe that the ideas for both the explicit IRP limiters and the high order IRP schemes studied in  \cite{Hyp16, JLPsy} can be readily extended to general hyperbolic conservation law systems (\ref{mcl}) as long as  (i) it features a global invariant region 
 \begin{equation}\label{s2}
\Sigma= \left\{ \mathbf{w}\big | \quad   U(\mathbf{w})\leq 0 \right\},  
\end{equation}
where $U$ is convex, and (ii) the corresponding one-dimensional projected system 
$$
\partial_t \mathbf{w}+\partial_\eta (\sum_{j=1}^d F_j(\mathbf{w})\nu_j)=0, \quad \eta \in \mathbb{R},
$$
where $\nu=(\nu_1, \cdots, \nu_d)$ is any unit vector,  shares the same invariant region $\Sigma$.  The later assumption is needed in order to obtain an IRP scheme. {These} observation{s} led to the present work on high order IRP schemes for general conservation law systems (\ref{mcl}).  This work may also be seen as to some degree an extension of the earlier works on positivity-preserving schemes for compressible Euler equations. {The present emphasis is on the notion of invariant regions. } 

In this paper we have the following objectives: \\
(i) to design an explicit IRP limiter, which can be shown to  maintain high order accuracy of the approximation;\\
(ii)  to identify sufficient conditions under which arbitrarily high order DG schemes 
 feature the desired IRP property 
for  both one and multi-dimensional hyperbolic conservation law systems; \\
(iii) to apply the general result in (ii) to two-dimensional compressible Euler equations, with numerical validations. 

As for (i),  our limiter for preserving $\Sigma$ in (\ref{s2}) is of the following form
\begin{align*}
\tilde{\mathbf{w}}_h(x)=\theta\mathbf{w}_h(x)+(1-\theta)\bar{\mathbf{w}}_h, \quad \theta =\min \left\{1,    \frac{U(\bar{\mathbf{w}}_h)}{U(\bar{\mathbf{w}}_h)-U^{\max}_h}\right\},
\end{align*}
where $\mathbf{w}_h(x)$ is a polynomial over domain $K$,  its average $\bar{\mathbf{w}}_h$ lies in the interior of $\Sigma$,  and $U^{\max}_h=\max \limits_{x\in K}U(\mathbf{w}_h(x))$. This reconstructed polynomial is shown to maintain the same order of accuracy as $\mathbf{w}_h(x)$,  \red{provided 
$U(\mathbf{w}_h(x))$ is not close to zero.}  

As for (ii), 
we present our analysis for general conservation law systems in order to illustrate the ways in which the IRP property can be ensured for high order finite volume type schemes. The first ingredient is a one-dimensional IRP numerical flux, such as  Godunov, Lax-Friedrichs, and Harten-Lax-van Leer \cite{HLL},  with which 
the first order finite volume scheme has the IRP property under certain CFL condition. This allows us to further find a sufficient condition to keep the cell averages in $\Sigma$ by repeated convex combinations, in the same way as that has been well established for positivity-preserving schemes, see e.g.  \cite{ZS10b, ZS12}. In the present setting, we use first order schemes with an IRP flux which can keep numerical solutions in $\Sigma$ as building blocks, and show that  high order spatial discretization with forward Euler can be written as  a convex combination of first order IRP schemes, thus will keep $\Sigma$ provided certain sufficient conditions are satisfied. The IRP limiter is then used to enforce the sufficient condition. 

For multi-dimensional conservation law systems we will show that repeated convex combinations can still be achieved, as long as a positive  convex decomposition for the cell averages is available, using the {Gauss}-Lobatto quadrature points on cell interfaces, and some interior points chosen so that the decomposition {weights are} strictly positive. The numerical solutions need to be within in $\Sigma$ only on a test set consisting of points also used for the cell average decomposition.   

As for (iii), we first show that the projected system of the two dimensional Euler systems indeed share the same invariant region $\Sigma$ in (\ref{s1}). The CFL conditions for the IRP DG schemes on rectangular and triangular meshes are derived, respectively, from our general result  for multi-dimensional hyperbolic conservation law systems, while using the test sets identified already in \cite{ZS10b} and \cite{ZXS12}.

Finally, we should mention that in our analysis we only show the ways of numerically preserving $\Sigma$ for the forward Euler time discretization, yet 
the high order SSP time discretizations  (\cite{SO88}) will keep the validity of our results since they are convex combinations of forward Euler.

This paper is organized as follows: in Sect. 2  we present an explicit invariant-region-preserving (IRP) limiter and prove that for smooth solutions the order of approximation accuracy is not destroyed by the IRP limiter in general cases, followed by a generic IRP algorithm for high order schemes.  Then, in Sect. 3, we first show all three popular numerical fluxes can be made an IRP flux, and then identify sufficient conditions including both a test set and the CFL condition, to obtain IRP DG schemes in one and higher space dimensions for arbitrary shape-regular meshes. Sect. 4  is devoted to an application to high order DG schemes for two dimensional  compressible Euler equations. In Sect. 5, we present  extensive numerical tests. Some concluding remarks are given in Sect. 6.  In Appendix A, we prove Lemma 2.1 for the compressible Euler equations for which the pressure is not strictly concave. In Appendix B, we present the detailed proof of Lemma 3.2 which states that HLLC flux is an IRP flux.

\section{The invariant-region-preserving limiter}
For the general multi-dimensional system of conservation laws (\ref{mcl}), we assume it admits  an invariant region $\Sigma$, which is a convex set in the phase space $\mathbb{R}^l$, characterized by  
\begin{align}\label{IR}
\Sigma=\{\mathbf{w}\big | \quad U(\mathbf{w})\leq 0\}
\end{align}
with $U$ being convex.  In what follows, we use 
$$
\Sigma _0=\{\mathbf{w}\big | \quad U(\mathbf{w})< 0\}
$$
to denote the interior of $\Sigma$.

For any bounded domain $K$, we define the average of $\mathbf{w}(x)$ by
\begin{align*}
\bar{\mathbf{w}}=\frac{1}{|K|}\int _K\mathbf{w}(x)dx,
\end{align*}
where $|K|$ is the measure of $K$. The following lemma shows that such an averaging operator is a contraction, which enables us to use the cell average as a reference to construct the IRP limiter.

\begin{lem} \label{U+} 
Let $\mathbf{w}(x)$ be non-trivial \red{piecewise continuous vector functions}.  
If $\mathbf{w}(x)\in \Sigma$ for all $x\in K\subset \mathbb{R}^d$, and $U$ is strictly convex,  then $\bar{\mathbf{w}}\in \Sigma_0$ for any bounded domain $K$.
\end{lem}
\begin{proof}
Since $U$ is convex, using Jensen's inequality and the assumption, we have
\begin{align*}
U(\bar{\mathbf{w}})=U\left(\frac{1}{|K|}\int _K\mathbf{w}(x)dx\right)\leq \frac{1}{|K|}\int _K U(\mathbf{w}(x))dx\leq 0.
\end{align*}
With this, we can show $U(\bar{\mathbf{w}})<0$. Otherwise if $U(\bar{\mathbf{w}})=0$, we must have $U(\mathbf{w}(x))=0$ for \red{almost} all $x\in K $; that is
\begin{align*}
U(\bar{\mathbf{w}})=U(\mathbf{w}(x))\quad \red{a.e. \; {\rm in}\; K.}
\end{align*}
This,  upon taking cell average on both sides,  gives
\begin{align*}
U(\bar{\mathbf{w}})=\frac{1}{|K|}\int _K U(\mathbf{w}(x))dx.
\end{align*}
By taking the Taylor expansion around $\bar{\mathbf{w}}$, we have
\begin{align*}
U(\mathbf{w}(x))=U(\bar{\mathbf{w}})+\triangledown _{\mathbf{w}}U(\bar{\mathbf{w}})\cdot \xi+\xi ^\top H\xi, \quad \forall x\in K, \quad \xi:=\mathbf{w}(x)-\bar{\mathbf{w}},
\end{align*}
which upon integration yields $\frac{1}{|K|}\int _K\xi ^\top H\xi dx=0$, where $H$ is the Hessian matrix of $U$. This when combined with the strict convexity of $U$ ensures that $\mathbf{w}(x)\equiv \bar{\mathbf{w}}$ almost everywhere, which contradicts the assumption. 
\end{proof}


\subsection{The Limiter}
Let $\mathbf{w}_h(x)$ be a sequence of vector polynomials over $K$, which is a high order accurate approximation to the smooth function $\mathbf{w}(x)\in \Sigma$. We assume $\bar{\mathbf{w}}_h\in \Sigma _0$, but $\mathbf{w}_h(x)$ is not entirely located in $\Sigma$, then we can modify the polynomial $\mathbf{w}_h(x)$ with reference to $\bar{\mathbf{w}}_h$
through a linear convex combination:
\begin{align}\label{modp}
\tilde{\mathbf{w}}_h(x)=\theta\mathbf{w}_h(x)+(1-\theta)\bar{\mathbf{w}}_h,
\end{align}
where $\theta \in (0,1]$ is defined by 
\begin{align}\label{limiter}
\theta =\min \{1,\theta _1\},
\end{align}
where
\begin{align}\label{theta}
\theta _1=\frac{U(\bar{\mathbf{w}}_h)}{U(\bar{\mathbf{w}}_h)-U^{\max}_h},
\end{align}
with 
\begin{align}\label{minmax}
U^{\max}_h=\max \limits_{x\in K}U(\mathbf{w}_h(x))>0.
\end{align}
Notice that since $\bar{\mathbf{w}}_h\in \Sigma _0$, we have $U(\bar{\mathbf{w}}_h)<0$. Also $U(\bar{\mathbf{w}}_h)<U^{\max}_h$. 
Therefore, $\theta _1$ is well-defined and positive.

As for the above limiter,  we have the following conclusion.


\begin{thm}  \label{th2.2}If $\bar{\mathbf{w}}_h\in \Sigma _0$, then $\tilde{\mathbf{w}}_h(x)\in \Sigma$, for all $x\in K$.
Moreover, the reconstructed polynomial preserves high order accuracy, i.e.,  if \red{$\| \mathbf{w}_h-\mathbf{w}\| _{\infty} \leq 1$}, then 
\begin{align*}
\red{
\| \tilde{\mathbf{w}}_h-\mathbf{w}\|_{\infty}\leq \frac{C}{|U(\bar{\mathbf{w}}_h)|}\| \mathbf{w}_h-\mathbf{w}\| _{\infty},
}
\end{align*}
where $C>0$ depends on $\mathbf{w}$ and $\Sigma$.
\end{thm}
\begin{proof}  The claim that the constructed polynomial lies within $\Sigma$ is implied by the definition of $\theta$.   In fact, for the case $\theta=\theta_1$ with the convexity of $U$, we have
\begin{align*}
U(\tilde{\mathbf{w}}_h(x))\leq &\theta  U(\mathbf{w}_h(x))+(1-\theta)U(\bar{\mathbf{w}}_h)\\
\leq & \theta_1  U_h^{\max}+(1-\theta_1)U(\bar{\mathbf{w}}_h)=0.
\end{align*}
For the accuracy estimate,  we consider the case when $\theta \neq 1$. We only need to prove
\begin{align}\label{2.6}
\| \tilde{\mathbf{w}}_h-\mathbf{w}_h\| _{\infty}\leq \frac{C}{|U(\bar{\mathbf{w}}_h)|} \| \mathbf{w}_h -\mathbf{w}\| _{\infty},
\end{align}
from which the conclusion follows by using the triangle inequality. 
From the reconstruction, it follows that
\begin{align*}
\| \tilde{\mathbf{w}}_h-\mathbf{w}_h\|_{\infty} = &(1-\theta)\| \bar{\mathbf{w}}_h-\mathbf{w}_h\| _{\infty}\\
=&\frac{\| \bar{\mathbf{w}}_h-\mathbf{w}_h\|_{\infty}}{U^{\max}_h-U(\bar{\mathbf{w}}_h)}U^{\max}_h.
\end{align*}
Since $U(\mathbf{w})\leq 0$ for $\mathbf{w}(x)\in \Sigma$, we have
\begin{align*}
U^{\max}_h\leq \max _{x\in K}\left(U(\mathbf{w}_h))-U(\mathbf{w}) \right)\leq \| \triangledown U\|_{\infty}\| \mathbf{w}-\mathbf{w}_h\| _{\infty}.
\end{align*}
Also, since $U^{\max}_h>0$ when $\theta < 1$, we have
\begin{align*}
U^{\max}_h-U(\bar{\mathbf{w}}_h)> -U(\bar{\mathbf{w}}_h)>0.
\end{align*}
According to the assumption that $\mathbf{w}_h$ is an approximation to $\mathbf{w}$, we have
\begin{align*}
\| \bar{\mathbf{w}}_h-\mathbf{w}_h\|_{\infty} =  \| \bar{\mathbf{w}}_h -\bar{\mathbf{w}}+\bar{\mathbf{w}}-\mathbf{w}+\mathbf{w}-\mathbf{w}_h\|_{\infty}\leq 2\| \mathbf{w}-\mathbf{w}_h\| _{\infty}+\| \mathbf{w}-\bar{\mathbf{w}}\|_{\infty}
\end{align*}
Therefore,  we arrive at (\ref{2.6}) with  $C$ given by
\red{
\begin{align*}
C =2 \|\triangledown U\|_{\infty} \left( 1+\| \mathbf{w}\|_{\infty}\right),
\end{align*}
}
which is positive and finite, depending only on $\mathbf{w}$,  as well as the invariant region $\Sigma$ through $U$.
\end{proof}

\begin{rem}
We would like to point out that when $\bar{\mathbf{w}}_h$ is close enough to the boundary of $\Sigma$, the factor $C/|U(\bar{\mathbf{w}}_h)|$ can become large, indicating the possibility of accuracy deterioration in some cases.
\end{rem}

In practice, $\Sigma$ is usually given by several pieces of convex functions in the form of
\begin{align*}
\Sigma=\bigcap \limits^M_{i=1}\{\mathbf{w}\big | \quad U_i(\mathbf{w})\leq 0\}.
\end{align*}
Then the limiter parameter given in (\ref{theta}) needs to be modified as 
\begin{align}\label{theta1}
\theta =\min \{1,\theta _1, \cdots ,\theta _M\},
\end{align}
where
\begin{align}\label{minmax1}
\theta _i =\frac{U_i(\bar{\mathbf{w}}_h)}{U_i(\bar{\mathbf{w}}_h)-U^{\max}_{i,h}}, \quad U^{\max}_{i,h}=\max _{x\in K}U_i(\mathbf{w}_h(x)).
\end{align}
It can be shown that Theorem \ref{th2.2} remains valid for the general case as such. 

Here we present two such examples for a convex invariant region. \\

\noindent{\bf Example 1.} The first example is the one dimensional isentropic gas dynamic system in Euler coordinates, i.e. the system (\ref{euler}) with $\mathbf{w}=(\rho, m)^\top$ and $f(\mathbf{w})=(\rho u, \rho u^2+p(\rho))^\top$, where $p(\rho)=\rho ^\gamma$, $\gamma >1$ and $m=\rho u$. The corresponding invariant region is given by
\begin{align*}
\Sigma =\{(\rho, m)^\top \big | \quad r\leq r_0, s\geq s_0\},
\end{align*}
where $r_0=\sup \limits_xr(\rho _0(x),m_0(x))$, $s_0=\inf \limits_xs(\rho _0(x),m_0(x))$ and
\begin{align*}
r=u+\frac{2\sqrt{\gamma}}{\gamma -1}\rho ^{\frac{\gamma -1}{2}}, \quad s=u-\frac{2\sqrt{\gamma}}{\gamma -1}\rho ^{\frac{\gamma -1}{2}},
\end{align*}
are two Riemann invariants. {We point out that $\Sigma$  is a closed domain in $\{(\rho, m)^\top| \; \rho\geq 0\}$}.  \\

\red{Note for the pressure-less Euler equation,  i.e. the system (\ref{euler}) with $\mathbf{w}=(\rho, m)^\top$ and $f(\mathbf{w})=(\rho u, \rho u^2)^\top$, where $m=\rho u$,  two Riemann invariants $r$ and $s$ become identical. In such case, the invariant region is  
\begin{align}\label{pl}
\Sigma=\{ (\rho, m)^\top \big|\quad  \rho>0, \;  s_0\rho \leq m\leq r_0 \rho\},
\end{align}
where $r_0=\sup \limits_x (u_0(x))$ and $s_0=\inf \limits_x (u_0(x)).$  
 }\\

\noindent{\bf Example 2.}  Another example is the compressible Euler equations, for which an invariant region is 
$$
\Sigma=\{ (\rho, \mathbf{m}, E)^\top \big|\quad \rho>0, \; p>0, \;q \leq 0\},
$$
with $p$ and $q$ defined by 
$$
p=(\gamma-1) (E-\frac{1}{2}\rho|\mathbf{u}|^2), \quad q=(s_0-s)\rho,\; \gamma>1,
$$
where $\mathbf{u}$ is the velocity, $\mathbf{m}=\rho \mathbf{u}$, $s=\log \left(\frac{p(x)}{\rho ^{\gamma}(x)} \right)$, $s_0=\inf \limits_x\log \left(\frac{p_0(x)}{\rho ^{\gamma}_0(x)} \right)$, and $(\rho _0, \mathbf{m}_0,E_0)^\top$ is the initial data. 

\begin{rem} When $U$ is not strictly convex, the result in Lemma \ref{U+} may still hold true. The proof needs to be modified based on further details of $U$.    For example,  for the compressible Euler equation, $p$ is concave but not strictly concave. We present an illustrative proof in Appendix A.    
\end{rem}

\subsection{Algorithm} 
Let $\mathbf{w}^{n}_h$ be  the numerical solution at $n$-th  time step generated from a high order scheme of an abstract form  
\begin{align}\label{adg}
\mathbf{w}^{n+1}_h = \mathcal{L}(\mathbf{w}^n_h), 
\end{align} 
starting with initial data $\mathbf{w}^0_h$, where $\mathbf{w}^n_h=\mathbf{w}^n_h(x)\in V_h$, and $V_h$ is a finite element space of piecewise polynomials of degree $k$ in each computational cell $K$, i.e.,
\begin{align*}
V_h=\{v: v|_{K}\in \mathbb{P}^k(K)\}.
\end{align*} 
Assume $\lambda =\frac{\Delta t}{|K|}$ is the mesh ratio. The IRP algorithm can be stated as follows: 

\begin{alg}\label{Algo} 
Provided that scheme (\ref{adg}) has the following property:  there exists $\lambda_0$, and a test set $S$ such that if 
\begin{align*}
\lambda \leq \lambda_0 \quad \text{and} \quad \mathbf{w}^n_h(x) \in \Sigma \quad \text{for} \quad x\in S
\end{align*}
then
\begin{align*}
 \bar{ \mathbf{w}}^{n+1}_h \in  {\Sigma}_0;
\end{align*}
then the IRP limiter can be applied, with $K$ replaced by $S_K:=S\cap K$ in (\ref{minmax1}), i.e.,
\begin{align}\label{mthe}
U^{\max}_h=\max _{x\in S_K}U(\mathbf{w}_h(x)),
\end{align}
through the following algorithm: \\

\textbf{Step 1:}  Initialization: take the piecewise $L^2$ projection of $\mathbf{w}_0$ onto $V_h$, such that
\begin{align*}
\langle \mathbf{w}^0_{h}-\mathbf{w}_0, \phi\rangle =0, \quad \forall \phi \in V_h.
\end{align*}

\textbf{Step 2:}  Imposing the modified limiter (\ref{modp}), (\ref{limiter}) with (\ref{mthe}) on ${\mathbf{w}}^n_h$ for $n=0,1,\cdots$ to obtain $\tilde{\mathbf{w}}^n_h$.

\textbf{Step 3:} Update by the scheme: 
\begin{align*}
\mathbf{w}^{n+1}_h=\mathcal{L}(\tilde{\mathbf{w}}_h^n).
\end{align*}
\\
Return to \textbf{Step 2}.
\end{alg}

\begin{rem}
For given initial data $\mathbf{w}_0 \in \Sigma$, its cell average lies strictly within $\Sigma_0$. On the other hand, $\mathbf{w}^0_h$ may not lie entirely in $\Sigma$, but it has the same cell average as the initial data due to the $L^2$ projection. Therefore, the IRP limiter can already be applied to $\mathbf{w}^0_h$ (included in Step 2 in the algorithm). 
\end{rem}

\section{IRP DG schemes}
In this section, we discuss some sufficient conditions for high order DG schemes solving the general conservation laws to be invariant-region-preserving.  

\subsection{One dimensional case} We begin with the one-dimensional system of conservation laws of the form 
\begin{align}\label{1dcl}
 \partial_t \mathbf{w} + \partial_x f(\mathbf{w}) =0, 
\end{align}
where $f$ is a smooth vector flux function.  
A first order finite volume scheme on a cell $I_j=[x_{j-1/2}, x_{j+1/2}]$ takes the following form  
\begin{align}\label{1stFV}
\mathbf{w}^{n+1}_j=\mathbf{w}^n_j-\lambda\left(\hat{f}_{j+1/2}-\hat{f}_{j-1/2} \right),
\end{align}
where $\mathbf{w}^n_j$ is the approximation to the average of $\mathbf{w}(x)$ on $I_j=[x_{j-\frac{1}{2}},x_{j+\frac{1}{2}}]$ at $n$-th time level $t^n$.  
$\hat{f}_{j+1/2}$ is a single-valued numerical flux at the element interface, depending on the values of numerical solution from both sides
$$
\hat{f}_{j+1/2} = \hat{f}(\mathbf{w}^n_j,\mathbf{w}^n_{j+1}).
$$
In general, $\hat{f}_{j+1/2}$ is derived from some Riemann solvers (exact or approximate).  

\begin{defn} A consistent numerical flux $\hat{f}_{j+1/2}$ is called an IRP flux for (\ref{1dcl}) if there exists $c_0$, such that for $\sigma \lambda \leq c_0$, $\mathbf{w}^n_j$, $\mathbf{w}^n_{j\pm 1}\in \Sigma$ implies  $\mathbf{w}^{n+1}_j\in \Sigma _0$, where $\sigma$ is the global maximum wave speed of the system (\ref{1dcl}).
\end{defn}

For scalar conservation laws, the invariant region is simply an interval ensured by the maximum principle.  It is known that the monotone flux is maximum-principle-preserving, {see e.g. \cite{ZS10a}},  therefore it is also the IRP flux.  For systems, most popular numerical fluxes rely on Riemann solvers, which exactly compute or approximate the solution of the Riemann problem, i.e., (\ref{1dcl}) with initial data, 
\begin{align}\label{re}
\mathbf{w}(x, 0)=
\left\{\begin{array}{ll}
\mathbf{w}_l, & x<0\\ 
\mathbf{w}_r, & x>0.
\end{array}
\right.
\end{align}
The solution of the Riemann problem is self-similar. Assume that the Riemann solver also has some self-similar structure and is denoted by 
$R(\xi; \mathbf{w}_l, \mathbf{w}_r)$ with $\xi=\frac{x}{t}$.  Let $\sigma_l$ and $\sigma_r$ be the leftmost
and rightmost wave speed such that $R=\mathbf{w}_l$ for $\xi \leq \sigma_l$ and $R=\mathbf{w}_r$ for $\xi \geq  \sigma_r$. Let $S_l \leq \min\{\sigma_l, 0\}$ and $S_r \geq \max\{\sigma_r, 0\}$.  Integration of (\ref{1dcl})  over $[S_lt, S_rt] \times [0, t]$, divided by $(S_r-S_l)t$,  leads to the following 
identity
\begin{align}\label{id}
\frac{1}{S_r-S_l}\int_{S_l}^{S_r} {R}(\xi; \mathbf{w}_l, \mathbf{w}_r) d\xi =\frac{S_r\mathbf{w}_r-S_l\mathbf{w}_l}{S_r-S_l}- \frac{f_r-f_l}{S_r-S_l},
\end{align}
where $f_r=f(\mathbf{w}_r)$ and $f_l=f(\mathbf{w}_l)$. This identity is useful in finding sufficient conditions for each of the following numerical fluxes to be an IRP flux. 
\begin{enumerate}
\item {Godunov flux}: 
\begin{align}\label{Gdnv}
\hat{f}(\mathbf{w}_l,\mathbf{w}_r)=f(R(0; \mathbf{w}_l,\mathbf{w}_r));
\end{align}
\item Lax-Friedrich flux:
\begin{align}\label{gLF}
\hat{f}(\mathbf{w}_l,\mathbf{w}_r) =\frac{1}{2}\left(f(\mathbf{w}_l)+f(\mathbf{w}_r)-\sigma (\mathbf{w}_r-\mathbf{w}_l) \right);
\end{align}
\item HLL flux \cite{HLL}:
\begin{align}\label{HLL}
\hat{f}(\mathbf{w}_l,\mathbf{w}_r)=
\begin{cases}
f(\mathbf{w}_l), \quad &\text{if }0\leq \sigma_l, \\
\frac{\sigma _rf(\mathbf{w}_l)-\sigma _lf(\mathbf{w}_r)+\sigma _l\sigma _r(\mathbf{w}_r-\mathbf{w}_l)}{\sigma _r - \sigma _l}, \quad &\text{if }\sigma _l\leq 0\leq \sigma _r, \\
f(\mathbf{w}_r), \quad &\text{if }0\geq \sigma _r.
\end{cases}
\end{align}
\end{enumerate}
\begin{lem}\label{lem:IRPflux}
\noindent
\begin{itemize}
\item[(i)] For $c_0=1$,  both Godunov flux and Lax-Friedrich flux are IRP fluxes for (\ref{1dcl});
\item[(ii)] For $c_0=\frac{1}{2}$, the HLL flux is an IRP flux for  (\ref{1dcl}).
\end{itemize}
\end{lem}
\begin{proof}
(i) With the Godunov flux in (\ref{1stFV}),  
$\mathbf{w}^{n+1}_j$ can be viewed as the cell average of the exact Riemann solution at $t^{n+1}$ when  $\lambda \sigma \leq 1$. 
See \cite[Section 13.2]{CLLeVeque}. Since the exact solution lies in $\Sigma$, then according to Lemma \ref{U+}, we have $\mathbf{w}^{n+1}_j\in \Sigma _0$ if $\lambda \sigma\leq 1$.

When the Lax-Friedrich flux (\ref{gLF}) is used,
the update $\mathbf{w}^{n+1}_j$ in (\ref{1stFV}) can be rewritten as 
\begin{align*}
\mathbf{w}^{n+1}_j=\left(1-\lambda \sigma \right)\mathbf{w}^n_j+\lambda \sigma\mathbf{w}^*,
\end{align*}
where
\begin{align*}
{\mathbf{w}}^*= \left(\frac{\mathbf{w}^n_{j-1}+\mathbf{w}^n_{j+1}}{2}-\frac{f(\mathbf{w}^n_{j+1})-f(\mathbf{w}^n_{j-1})}{2\sigma} \right).
\end{align*}
From (\ref{id}) it follows that 
\begin{align*}
{\mathbf{w}}^*=\frac{1}{2\sigma }\int ^{\sigma }_{-\sigma}
R(\xi; \mathbf{w}^n_{j-1}, \mathbf{w}^n_{j+1})d\xi.
\end{align*}
For $\mathbf{w}^n_{j\pm1} \in \Sigma$, we have $R(\xi; \mathbf{w}^n_{j-1},\mathbf{w}^n_{j+1}) \in \Sigma$. Therefore ${\mathbf{w}}^*$ lies in $\Sigma _0$ by Lemma \ref{U+}.  Since $\mathbf{w}^{n+1}_j$ is a convex combination of two vectors: $\mathbf{w}^n_j \in \Sigma $ and ${\mathbf{w}}^* \in \Sigma_0 $ for $\lambda \sigma\leq 1$, we then have $\mathbf{w}^{n+1}_j \in \Sigma_0$.

(ii) 
For the HLL flux (\ref{HLL}),
the evolved cell average $\mathbf{w}^{n+1}_j$ can be rewritten as
\begin{align}\label{HLLw}
\mathbf{w}^{n+1}_j=(1-\theta _1-\theta _2)\mathbf{w}^n_j+\theta _1\hat{\mathbf{w}}^1+\theta _2\hat{\mathbf{w}}^2,
\end{align}
with
\begin{align*}
\hat{\mathbf{w}}^1=\frac{b_1\mathbf{w}^n_{j+1}-a_1\mathbf{w}^n_j}{b_1 -a_1}-\frac{f(\mathbf{w}^n_{j+1})-f(\mathbf{w}^n_j)}{b_1 -a_1},\\\quad \hat{\mathbf{w}}^2=\frac{b_2\mathbf{w}^n_j-a_2\mathbf{w}^n_{j-1}}{b_2-a_2}-\frac{f(\mathbf{w}^n_j)-f(\mathbf{w}^n_{j-1})}{b_2-a_2},
\end{align*}
where $\theta _1=-\lambda a_1$, $\theta _2=\lambda b_2$ with
\begin{align*}
b_1= \max\{\sigma _{j+\frac{1}{2},r},0\}, \quad a_1=\min\{\sigma _{j+\frac{1}{2},l},0\},\\
b_2= \max\{\sigma _{j-\frac{1}{2},r},0\}, \quad a_2=\min\{\sigma _{j-\frac{1}{2},l},0\},
\end{align*}
and $\sigma _{j+\frac{1}{2},l}$ and $\sigma _{j+\frac{1}{2},r}$ are the leftmost and rightmost wave speeds at $x_{j+\frac{1}{2}}$.
Notice that both $\hat{\mathbf{w}}^1$ and $\hat{\mathbf{w}}^2$ are in the form of (\ref{id}), the cell average of some exact Riemann solutions,
hence they both lie in $\Sigma _0$ by Lemma \ref{U+}. Therefore $\lambda \sigma \leq \frac{1}{2}$ is a sufficient condition for $\mathbf{w}^{n+1}_j$ in (\ref{HLLw}) to be in $\Sigma _0$.
\end{proof}
\begin{rem}
Notice that the local Lax-Friedrich flux is a special case of HLL flux, where
\begin{align*}
\sigma _{j+\frac{1}{2},r}=-\sigma _{j+\frac{1}{2},l}=\max _{{\mathbf{w}^n_j,\mathbf{w}^n_{j+1}}}|\partial_{\mathbf{w}} f(\cdot)|.
\end{align*}
Hence the local Lax-Friedrich flux is an IRP flux when $\lambda\sigma \leq \frac{1}{2}$. Here we use $|\partial_{\mathbf{w}} f|$ as a notation to denote the absolute value of eigenvalues of Jacobian matrix $\partial_{\mathbf{w}} f $. 
\end{rem}
\red{
The HLLC approximate Riemann solver as a three wave model was proposed by Toro et al. \cite{TSS94} as a modification of the HLL flux whereby the missing contact and shear waves in the Euler equations are restored. The HLLC flux is given by
\begin{align*}
\hat{f}(\mathbf{w}_l,\mathbf{w}_r)=
\begin{cases}
f(\mathbf{w}_l), \quad \text{if }0\leq \sigma _l,\\
f_{*l}, \quad \text{if }\sigma _l\leq 0\leq \sigma _*,\\
f_{*r}, \quad \text{if }\sigma _*\leq 0\leq \sigma _r,\\
f(\mathbf{w}_r), \quad \text{if }0\geq \sigma _r,
\end{cases}
\end{align*}
where $\sigma _{*}$ is the speed of middle wave, and the intermediate fluxes are given by
\begin{align*}
f_{*l}=f(\mathbf{w}_l)+\sigma _l(\mathbf{w}_{*l}-\mathbf{w}_l), \quad f_{*r}=f(\mathbf{w}_r)+\sigma _r(\mathbf{w}_{*r}-\mathbf{w}_r),
\end{align*}
and $\mathbf{w}_{*l}$, $\mathbf{w}_{*r}$ are two intermediate states determined by integral averages of the Riemann solution 
$$
\mathbf{w}_{*l}=\frac{1}{\sigma_*-\sigma_l}\int_{\sigma_l}^{\sigma_*}\mathbf{w}(\xi t, t)d\xi, \quad \mathbf{w}_{*r}=\frac{1}{\sigma_r- \sigma_*}\int_{\sigma_*}^{\sigma_r}\mathbf{w}(\xi t, t)d\xi.
$$
The two intermediate fluxes are related by 
\begin{align}\label{flfr}
f_{*r}=f_{*l}+\sigma _*(\mathbf{w}_{*r}-\mathbf{w}_{*l}).
\end{align}
Note that there are more unknowns than equations and some extra conditions need to be imposed in order to determine the intermediate fluxes, see \cite{TSS94}
for two versions of the HLLC flux for the compressible Euler equation.  For general 1D hyperbolic conservation law systems, the following result for the HLLC flux to be an IRP flux is proved in Appendix B. 
\begin{lem}\label{lemHLLC}
For $c_0=\frac{1}{2}$, the HLLC flux is an IRP flux.
\end{lem}
}
\red{
\begin{rem} We want to point out that Algorithm \ref{Algo} can still be applied to weakly hyperbolic conservation laws.  A canonical example is the pressure-less Euler system,  due to the formation of vacuum and/or delta-shock formation in the density,  care is needed in the choice of the numerical flux. The Godunov flux derived in \cite{BSL} was used in the DG scheme given in \cite{YWS}. From \cite[Lemma 4.2]{YWS}
we see that the Godunov flux  is indeed an IRP flux for $c_0=\frac{1}{2}$ with $\sigma=\max\{|s_0|, |r_0|\}$ to preserve $\Sigma$ as defined in (\ref{pl}). Our explicit IRP limiter can of course be used as an alternative to the special limiter constructed in \cite{YWS} in order to fulfill the two requirements: $\rho$ is positive and the velocity  $u=m/\rho$ satisfies a maximum principle.
\end{rem}
}

For a $(k+1)$th-order scheme with reconstructed polynomials or approximation polynomials of degree $k$, with forward Euler time discretization, the cell average evolves by
\begin{align}\label{eq: High order FV scheme}
\bar{\mathbf{w}}^{n+1}_j=\bar{\mathbf{w}}^n_j-\lambda[\hat{f}(\mathbf{w}^-_{j+\frac{1}{2}},\mathbf{w}^+_{j+\frac{1}{2}})-\hat{f}(\mathbf{w}^-_{j-\frac{1}{2}},\mathbf{w}^+_{j-\frac{1}{2}})],
\end{align}
where $\bar{\mathbf{w}}^n_j$ is the cell average of $ \mathbf{w}^{n}_h$ on $I_j$ at time level $n$, $\mathbf{w}^{\pm }_{j+\frac{1}{2}}$ are approximations to the point value of $\mathbf{w}$ at $x_{j+1/2}$ at time level $n$ from the left and the right cells,  respectively.

We consider an $N-$point 
Legendre Gauss-Lobatto quadrature rule on $I_j$,  with quadrature weights $\hat \omega_i$ on $[-\frac{1}{2},\frac{1}{2}]$ such that  $\sum_{i=1}^N \hat \omega_i =1$, which is exact for integrals of polynomials of degree up to $k$, if $2N-3 \geq k$. 
Denote these quadrature points  on $I_j$ as 
$$
S_j:=\{\hat x^i_j, 1\leq i \leq N\},  
$$
where $\hat x^1_j=x_{j-1/2}$ and $\hat x^N_j=x_{j+1/2}$.  The cell average decomposition then takes the form 
\begin{align}\label{avgdec-}
\bar{\mathbf{w}}^n_j=\sum_{i=2}^{N-1} \hat\omega_i \mathbf{w}_h^n(\hat x_j^i) +\hat \omega_1
\mathbf{w}^{+}_{j-\frac{1}{2}} +\hat \omega_N \mathbf{w}^{-}_{j+\frac{1}{2}},
\end{align}
where it is known that $\hat \omega_1=\hat \omega_N=1/(N(N-1))$.  Hence  (\ref{eq: High order FV scheme}) can be rewritten as a linear convex combination of the form 
\begin{align}\label{decschm}
\bar{\mathbf{w}}^{n+1}_j=\sum ^{N-1}_{i=2}\hat{\omega}_{i}\mathbf{w}_h^n(\hat{x}^i_j)+\hat{\omega}_1H_1 +\hat{\omega}_NH_N,
\end{align}
where 
\begin{align*}
H_1=\mathbf{w}^+_{j-\frac{1}{2}}-\frac{\lambda}{\hat{\omega}_1}\left(\hat{f}(\mathbf{w}^+_{j-\frac{1}{2}},\mathbf{w}^-_{j+\frac{1}{2}})-\hat{f}(\mathbf{w}^-_{j-\frac{1}{2}},\mathbf{w}^+_{j-\frac{1}{2}}) \right),\\
H_N=\mathbf{w}^-_{j+\frac{1}{2}}-\frac{\lambda}{\hat{\omega}_N}\left(\hat{f}
(\mathbf{w}^-_{j+\frac{1}{2}},\mathbf{w}^+_{j+\frac{1}{2}})- \hat{f}(\mathbf{w}^+_{j-\frac{1}{2}},\mathbf{w}^-_{j+\frac{1}{2}})\right)
\end{align*}
are of the same type as the first order scheme (\ref{1stFV}).  The decomposition of (\ref{decschm}) is first introduced by Zhang and Shu (\cite{ZS10b}) for the compressible Euler equation and it suffices for us to conclude the following result. 
\begin{thm}[High order scheme]\label{thm:highFV} 
A sufficient condition for $ \bar{\mathbf{w}}^{n+1}_j\in \Sigma _0$ by scheme (\ref{eq: High order FV scheme}) with an IRP flux is 
$$
\mathbf{w}^n_h(x)\in \Sigma \quad \text{for} \; x\in  S_j
$$
under the CFL condition 
\begin{align}\label{pscfl}
\sigma\lambda \leq \frac{1}{N(N-1)}c _0 \quad \text{with} \quad N=\lceil \frac{k+3}{2}\rceil,
\end{align}
where $\sigma$ is the {global maximum of wave speed}, $c_0$ is dependent on the IRP flux and $k$ is the degree of approximation polynomials.
\end{thm}

\subsection{Multi-dimensional case}
To solve (\ref{mcl}) over a computational  cell $K$, we consider a high order DG scheme. That is, to find $\mathbf{w}_h\in V_h$ such that  
\begin{align}\label{dg}
\int _{K}\partial_t (\mathbf{w}_h)\phi dx-\int _{K}\mathbf{F}(\mathbf{w}_h)\cdot \nabla_x \phi dx+ \sum ^Q_{i=1}\int _{e^i_K} \hat{F}(\mathbf{w}_h^-,\mathbf{w}_h^+, \nu ^i)\phi ds =0, \quad \forall \phi \in V_h,
\end{align}
where {$\mathbf{F}=(F_1,\cdots, F_d)^\top$}, $e^i_K$ is the $i$-th edge (or surface) of $K$, $\nu ^i$ is the normal vector on $e^i_K$, $\mathbf{w}^-_h$ and $\mathbf{w}^+_h$ denote the approximation to $\mathbf{w}_h$ on the edge of $K$ from interior and exterior of $K$ respectively, and $\hat{{F}}$ is an admissible numerical flux. 

Taking $\phi=1/|K|$ in this scheme, where $|K|$ is the area (or volume) of the element, and evaluate the interface integral with an appropriate quadrature rule, we see that cell averages are actually evolved, when using the Euler-forward for time discretization, by   
\begin{align}\label{cdgK2}
\bar{\mathbf{w}}^{n+1}_K=\bar{\mathbf{w}}^n_K-\frac{\Delta t}{|K|}\sum ^Q_{i=1}\sum ^L_{\beta =1}w_{\beta}\hat{F}(\mathbf{w}^{i,\beta}_{K},\mathbf{w}^{i,\beta}_{K_i},\nu ^i)|e_K^i|,
\end{align}
where 
$$
\bar{\mathbf{w}}^n_K=\frac{1}{|K|}\int _K\mathbf{w}^n_hdK,
$$
$\mathbf{w}^{i,\beta}_{K}$ and $\mathbf{w}^{i,\beta}_{K_i}$ are approximations at $\beta$-th quadrature point to solution values on $e^i_K$ from cell $K$ and $K_i$ respectively, $w_{\beta}$ are corresponding quadrature weights, and the number of quadrature points $L$ is chosen to achieve the desired accuracy.


Following \cite{PS96, ZS10b, ZXS12} in the study of positivity-preserving schemes, we seek to rewrite the evolved cell average in (\ref{cdgK2}) into a linear convex combination of terms which can then be shown to lie strictly within $\Sigma$. 

Assume one can construct an exact decomposition of the cell average:
\begin{align}\label{avgdec}
\bar{\mathbf{w}}^n_K=\sum ^P_{\alpha =1}c_{\alpha}\mathbf{w}^{\alpha}_K+\sum ^Q_{i=1}\sum ^L_{\beta =1}d_{\beta}\mathbf{w}^{i,\beta}_{K},
\end{align}
where $\mathbf{w}^{\alpha}_K$ are approximations to solution values at some interior points in $K$,  $c_{\alpha}$ and $d_{\beta}$ are positive weights satisfying 
\begin{align}\label{qwts}
\sum ^P_{\alpha =1}c_{\alpha}+Q\sum ^L_{\beta =1}d_{\beta}= 1.
\end{align}
This allows for the following reformulation 
\begin{align}\label{avgscm}
\nonumber \bar{\mathbf{w}}^{n+1}_K=&\sum ^P_{\alpha =1}c_{\alpha}\mathbf{w}^{\alpha}_K+\sum ^Q_{i=1}\sum ^L_{\beta =1}d_{\beta}\mathbf{w}^{i,\beta}_{K}-\frac{\Delta t}{|K|}\sum ^Q_{i=1}\sum ^L_{\beta =1}w_{\beta}\hat{F}(\mathbf{w}^{i,\beta}_{K},\mathbf{w}^{i,\beta}_{K_i},\nu ^i)|e^i_K|\\
=&\sum ^P_{\alpha =1}c_{\alpha}\mathbf{w}^{\alpha}_K+\sum ^L_{\beta =1}d_{\beta}\sum ^Q_{i=1}H_{i,\beta},
\end{align}
where
\begin{align*}
H_{i,\beta}=&\mathbf{w}^{i,\beta}_K -\frac{\Delta t |e^i_K| w_{\beta}}{|K|d_{\beta}}\left( \hat{F}(\mathbf{w}^{i,\beta}_K,\mathbf{w}^{i,\beta}_{K_i},\nu ^i)-\hat{F}(\mathbf{w}^{Q,\beta}_K,\mathbf{w}^{i,\beta}_{K},\nu ^i) \right), \quad i=1,\cdots,Q-1, \\
H_{Q,\beta}=&\mathbf{w}^{Q,\beta}_K-\frac{\Delta t w_{\beta}}{|K|d_{\beta}}\left(\hat{F}(\mathbf{w}^{Q,\beta}_K,\mathbf{w}^{Q,\beta}_{K_Q},\nu^Q)|e^Q_K| +\sum ^{Q-1}_{i=1}\hat{F}(\mathbf{w}^{Q,\beta}_K,\mathbf{w}^{i,\beta}_K,\nu ^i)|e^i_K| \right).
\end{align*}
Furthermore, 
{$H_{Q,\beta}$ can be rewritten as} 
\begin{align}\label{HQ}
\nonumber H_{Q,\beta}=&\sum ^{Q-1}_{i=1}\frac{|e^i_K|}{|\partial K|}\left(\mathbf{w}^{Q,\beta}_K-\frac{\Delta t w_{\beta}}{|K|d_{\beta}}|\partial K| \left(\hat{F}(\mathbf{w}^{Q,\beta}_K,\mathbf{w}^{i,\beta}_{K},\nu ^i)-\hat{F}(\mathbf{w}^{Q,\beta}_K,\mathbf{w}^{Q,\beta}_K,\nu ^i) \right) \right)\\
&+\frac{|e^Q_K|}{|\partial K| }
\left(\mathbf{w}^{Q,\beta}_K-\frac{\Delta t w_{\beta}}{|K|d_{\beta}} |\partial K| 
\left(\hat{F}(\mathbf{w}^{Q,\beta}_K,\mathbf{w}^{Q,\beta}_{K_Q},\nu ^Q)-\hat{F}(\mathbf{w}^{Q,\beta}_K,\mathbf{w}^{Q,\beta}_K,\nu ^Q) \right)
 \right),
\end{align}
where $|\partial K|=\sum ^Q_{i=1}|e^i_K|$. Here we have used the identity 
$$
\sum_{i=1}^Q |e^i_K| \hat{F}(\mathbf{w}^{Q,\beta}_K,\mathbf{w}^{Q,\beta}_K,\nu ^i) =\int_{\partial K} \mathbf{F}(\mathbf{w}^{Q,\beta}_K)\cdot \nu dS=
\int_{K} {\rm div}(\mathbf{F}(\mathbf{w}^{Q,\beta}_K))dx=0,
$$
where $\nu|_{e^i_K}=\nu^i$.  At this point it is clear that (\ref{avgscm}) is a linear convex combination of interior point values $\mathbf{w}^{\alpha}_K$, and quantities of the form
\begin{align}\label{fmw}
\mathbf{w}^*-c\left(\hat{F}(\mathbf{w}^*,\mathbf{w}_l,\nu)-\hat{F}(\mathbf{w}_r,\mathbf{w}^*,\nu)\right),
\end{align}
where $c$ is a constant and $\nu$ is a unit vector. Note that (\ref{fmw}) can be viewed as obtained from a formal first order scheme to one-dimensional system 
\begin{align}\label{eq:projeq}
\partial_t \mathbf{w}+\partial_\eta  (\mathbf{F}\cdot \nu)=0.
\end{align}
\red{Here $\mathbf{F}\cdot \nu:=\sum_{i=1}^d \mathbf{F_i}\cdot \nu_i$ is a vector flux.} Therefore, the cell average $\bar{\mathbf{w}}^{n+1}_K$ in (\ref{avgscm}) can be shown located in the invariant region under some CFL conditions, as long as we can show the system (\ref{eq:projeq}), also known as projected equations in \cite{Colella}, admits the same invariant region $\Sigma $ for all vectors $\nu \in \{\nu^i\}_{i=1}^Q$.  We thus have proved the following result. 
\begin{thm}\label{cfl2d}
Suppose there exists a positive quadrature rule such that (\ref{avgdec}) holds. If for $\nu \in \{\nu^i\}_{i=1}^Q$, (\ref{eq:projeq}) admits the same invariant region $\Sigma$, then a sufficient condition for $\bar{\mathbf{w}}^{n+1}_K\in \Sigma _0$ by scheme (\ref{cdgK2}) with an IRP flux is 
\begin{align*}
\mathbf{w}^n_h(x)\in \Sigma \quad \text{for }x \in S_K,
\end{align*}
under the CFL condition
\begin{align*}
\sigma\frac{\Delta t}{|K|} \leq \min _{\beta} \frac{d_{\beta}}{|\partial K|w_{\beta}}c_0, 
\end{align*}
where $\sigma =\max \limits_i\sigma _i$, $\sigma _i=\max |\partial_w( F\cdot \nu^i)|$, 
$S_K$ is the set of quadrature points over $K$, $d_{\beta}$ are positive weights such that (\ref{avgdec}) is held and $w_{\beta}$ are quadrature weights used in (\ref{cdgK2}), and $c_0=1$ or $1/2$ depending on the IRP flux used.
\end{thm}

\section{Application to compressible Euler equations}
In this section, we apply the obtained results to the DG schemes for solving the compressible Euler equations.
\subsection{1D case}
We briefly review the IRP limiter first introduced in \cite{Hyp16} for the one dimensional compressible Euler equations of the form (\ref{1dcl}) with
\begin{align}\label{1dEuler}
\mathbf{w}=(\rho, m, E)^\top, \quad f(\mathbf{w})=(m, \rho u^2+ p, (E+p)u)^\top, 
\end{align}
where $\rho$ is the density, $m=\rho u$ is the momentum, $E$ is the total energy, and $p$ is the pressure satisfying 
\begin{align*}
 E=\frac{1}{2}\rho u^2 +\frac{p}{\gamma-1}, \quad \gamma>1.
\end{align*}
It is known that the system (\ref{1dcl}) with (\ref{1dEuler}) has an invariant region:
\begin{align}\label{IRE}
\Sigma =\{\mathbf{w}\big |\quad \rho >0, \; p>0, \;  q\leq 0\},
\end{align}
where $q=(s_0-s)\rho$,
\begin{align*}
s={\rm log} \left( \frac{p(x)}{\rho^\gamma (x)}\right), \quad s_0=\inf_{x} {\rm log} \left( \frac{p_0(x)}{\rho_0^\gamma (x)}\right).
\end{align*}
Here $\rho_0$ and $p_0$ are obtained from the  given initial data  $\mathbf{w}_0=(\rho _0, m_0, E_0)^\top$.

 In numerical simulations, we use 
 the modified set of admissible states defined as
\begin{align*}
\Sigma ^{\epsilon}=\{ \mathbf{w}\big |\quad  \rho \geq \epsilon,  p\geq \epsilon, q\leq 0\},
\end{align*}
and its interior defined as
\begin{align*}
\Sigma ^{\epsilon}_0=\{ \mathbf{w}\big |\quad \rho > \epsilon,  p> \epsilon, q< 0\},
\end{align*}
where $\epsilon$ is a small positive number as the desired lower bound for density and pressure. 



The one dimensional limiter introduced in Section 2 can be applied such that the modified polynomial (\ref{modp}) lies entirely in $\Sigma ^{\epsilon}$ and is still a high order approximation to $\mathbf{w}(x)$, if $\theta$ is chosen as
\begin{align}\label{IRlimiter}
\theta =\min \{1,\theta _1, \theta _2, \theta _3\},
\end{align}
where 
\begin{align}\label{thetaEuler}
\theta _1=\frac{\bar{\rho}_h-\epsilon}{\bar{\rho}_h-\rho _{h,\min}}, \quad 
\theta _2=\frac{p(\bar{\mathbf{w}}_h)-\epsilon}{p(\bar{\mathbf{w}}_h)-p_{h,\min}}, \quad
\theta _3=\frac{q(\bar{\mathbf{w}}_h)}{q(\bar{\mathbf{w}}_h)-q_{h,\max}}
\end{align}
with
\begin{align}\label{ElMinMax}
\rho _{h,\min}=\min _{x\in K}\rho_h(x), \quad p_{h,\min}=\min _{x\in K}p(\mathbf{w}_h(x)), \quad q_{h,\max}=\max _{x\in K}q(\mathbf{w}_h(x)),
\end{align}
where $K$ may be chosen as $S_j (j=1, \cdots, N)$ in one dimensional case, dictated by Theorem \ref{thm:highFV}. Also in the CFL condition given in 
(\ref{pscfl}),  $\sigma =\| |u|+c\| _{\infty}$, where $c=\sqrt{\frac{\gamma p}{\rho}}$ is the sound speed.


\subsection{2D case}
Consider the two dimensional compressible Euler equations of the form 
\begin{align}\label{eq:2DEuler}
\partial_t \mathbf{w} +\triangledown \cdot \mathbf{F}=0,
\end{align}
with $\mathbf{w}=(\rho,m,n,E)^{\top}$, $\mathbf{F}(\mathbf{w})=(F_1(\mathbf{w}),F_2(\mathbf{w}))$, where
\begin{align}\label{fg2d}
&F_1(\mathbf{w})=(m,\rho u^2+p, \rho uv, (E+p)u)^{\top},\quad
F_2(\mathbf{w})=(n,\rho uv, \rho v^2+p, (E+p)v)^{\top}\\
&m=\rho u, \quad n=\rho v, \quad E=\frac{1}{2}\rho u^2+\frac{1}{2}\rho v^2+\frac{p}{\gamma -1}.
\end{align}
In order to apply Theorem \ref{cfl2d}, we first show that  (\ref{eq:projeq}) admits the same invariant region $\Sigma$.  
\red{The invariant region for weak solutions to hyperbolic conservation law systems  is in general an open problem due to lack of the well-posedness result.  As observed by Lax in \cite{La71}, Glimm's solutions for one-dimensional systems satisfy all relevant entropy conditions,  therefore, in \cite{Hoff85} necessary and sufficient conditions for a region to be invariant for (Glimm) solutions of the system of conservation laws are given. We shall verify Hoff's conditions, for which the building blocks are Riemann solutions, as needed in the present situation with $\Sigma$. 
}
\begin{lem}\label{lem:proj}
Let  $\bm{\nu}$ be any  unit vector, then system (\ref{eq:projeq}) with $\mathbf{w}=(\rho,m,n,E)^{\top}$ and $\mathbf{F}(\mathbf{w})=(F_1(\mathbf{w}),F_2(\mathbf{w}))$, where $F_1$ and $F_2$ are given in (\ref{fg2d}), has the same invariant region 
\begin{align*}
\Sigma =\{\mathbf{w} \big |\quad \rho >0, p>0, q\leq 0\}.
\end{align*}
\end{lem}
\begin{proof}
Let $u^N=(u,v)\cdot \bm{\nu}$ and $u^T=(u,v)\cdot \bm{\nu}^{\perp}$.  System (\ref{eq:projeq}) can be rewritten as
\begin{equation}\label{eq:1dsys}
\begin{aligned}
&\partial_t \rho +\partial_\eta (\rho u^N)=0, \\
&\partial_t (\rho u^N)+\partial_\eta (\rho (u^N)^2+p)=0,\\
&\partial_t (\rho u^T)+\partial_\eta  (\rho u^Nu^T)=0,\\
&\partial_t E +\partial_\eta ((E+p)u^N)=0, 
\end{aligned}
\end{equation}
where 
\begin{align*}
 p=(\gamma -1)(E-\frac{1}{2}\rho (u^N)^2-\frac{1}{2}\rho (u^T)^2)
\end{align*}
as deduced from $p=(\gamma -1)(E-\frac{1}{2}\rho u^2-\frac{1}{2}\rho v^2)$. \red{These equations using the primitive variables: $U=(\rho, u^N, u^T, p)^\top$ may be written as 
$$
\partial_t U +A(U) \partial_\eta U=0,
$$
where 
$$
A=
\left(
\begin{array}{cccc}
 u^N &  \rho & 0  & 0 \\
0  &   u^N& 0  &1/\rho \\
 0 & 0  &   u^N& 0 \\
0  & \gamma p &  0&u^N
\end{array}
\right).
$$
Its eigenvalues are $u^N-c, u^N, u^N, u^N+c$, where the speed of sound is  $c=\sqrt{\gamma p/\rho}$. The associated left eigenvectors are 
\begin{align*}
& l^1=(0, -\rho/2c, 0, 1/(2c^2)), \quad l^2=(1, 0,  0, -1/c^2), \\
& l^3=(1, 0,  1, -1/c^2), \quad l^4=(0, \rho/(2c), 0, 1/(2c^2)).
\end{align*}
Consider $\partial \Sigma=\{\mathbf{w} \big |\quad \rho >0, p>0, q=0\}$, its normal direction is 
$$
\vec{n}=(\gamma, 0,  0, -\gamma/c^2)=\gamma l^2, 
$$
and  $\Sigma$  is convex. This meets the sufficient and necessary conditions given in \cite[Corollary 3.3]{Hoff85} for the intersection with a half space, 
therefore }
\begin{align*}
\Sigma =\{(\rho,m,n,E)^\top \big | \rho >0, p>0, \; q\leq 0\}
\end{align*}
is an invariant region for (\ref{eq:1dsys}).
\end{proof}

We next identify the test sets $S_K$ as required in Theorem \ref{cfl2d}. In fact, the existing results on test sets for positivity-preserving DG schemes 
established in \cite{ZS10b, ZXS12} can still be used for the IRP DG schemes presented in this work. For approximation polynomials of degree $k$,  we discuss two kinds of meshes in the following.
\begin{enumerate}
\item[(i)] For rectangular mesh $K=[x_{i-\frac{1}{2}},x_{i+\frac{1}{2}}] \times [y_{j-\frac{1}{2}},y_{j+\frac{1}{2}}]$, the test set $S_K$  in \cite[Theorem 3.1]{ZS10b} is 
$$
S_K=(S^x_i\times \hat{S}^y_j)\cup (\hat{S}^x_i\times S^y_j),
$$ where
\begin{align*}
S^x_i=\{x^{\beta}_i, \beta =1, \cdots, L\}, \quad S^y_j=\{y^{\beta}_j, \beta =1, \cdots, L\}
\end{align*}
are the Gauss quadrature points on $[x_{i-\frac{1}{2}},x_{i+\frac{1}{2}}]$ and $[y_{j-\frac{1}{2}},y_{j+\frac{1}{2}}]$ respectively and $L$ is chosen such that the quadrature rule is exact for single variable polynomials of degree $2k+1$,  and
\begin{align*}
\hat{S}^x_i=\{\hat{x}^{\alpha}_i, \alpha =1, \cdots, N\}, \quad \hat{S}^y_j=\{\hat{y}^{\alpha}_j, \alpha =1, \cdots, N\}
\end{align*}
are the Gauss-Lobatto quadrature points on $[x_{i-\frac{1}{2}},x_{i+\frac{1}{2}}]$ and $[y_{j-\frac{1}{2}},y_{j+\frac{1}{2}}]$ respectively and $N$ is chosen such that $2N-3\geq k$.

\red{In the form of decomposition (\ref{avgdec}) satisfying (\ref{qwts}), the cell average $\bar{\mathbf{w}}^n_K$ is then given by}
\begin{align*}
\bar{\mathbf{w}}^n_K=\sum ^{N-1}_{\alpha =2}\sum ^L_{\beta =1}\frac{1}{2}w_{\beta}\hat{w}_{\alpha}\left(\mathbf{w}_h(x^{\beta}_i,\hat{y}^{\alpha}_j)+\mathbf{w}_h(\hat{x}^{\alpha}_i,y^{\beta}_j)\right)+\sum ^4_{l=1}\sum ^L_{\beta =1}\frac{1}{2}w_{\beta}\hat{w}_1\mathbf{w}^{l,\beta}_K,
\end{align*}
where $\mathbf{w}^{l,\beta}_K$ are values of approximation polynomial $\mathbf{w}_h$ in $K$ at $\beta$-th quadrature points on the $l$-th edge of $K$ and {$\hat{w}_1$ is the weight of first Gauss-Lobatto quadrature point and is equal to $\hat{w}_N$}. Then according to Theorem \ref{cfl2d}, we find the CFL condition for IRP DG schemes on rectangular meshes is 
\begin{align}\label{rcfl1}
\sigma \frac{\Delta t}{\Delta x\Delta y}\leq \frac{\hat{w}_1c_0}{4(\Delta x+\Delta y)},
\end{align}
where $\sigma=\max\{\||u|+c\|_{\infty},\||v|+c\|_{\infty}\}$ and $c_0$ is 1 or $\frac{1}{2}$ depending on the IRP flux used.

\begin{rem} For rectangular meshes, 
one could use a simple dimension by dimension decomposition for $\bar{\mathbf{w}}^{n+1}_K$, leading to a less restricted CFL condition 
on the time step, which is 
\begin{align}\label{rcfl2}
\Delta t=\frac{\hat{w}_1c_0}{\frac{\sigma _1}{\Delta x}+\frac{\sigma _2}{\Delta y}},
\end{align}
as obtained in \cite[(3.12)]{ZS10b}.
 
\end{rem}

\item[(ii)] For triangular meshes,  the authors in  \cite[(3.3)]{ZXS12} introduced the following set of quadrature points, denoted by barycentric coordinates, as 
\begin{align*}
S _K=\bigg \{ &\left(\frac{1}{2}+v^{\beta},(\frac{1}{2}+\hat{u}^{\alpha})(\frac{1}{2}-v^{\beta}),(\frac{1}{2}-\hat{u}^{\alpha})(\frac{1}{2}-v^{\beta})\right),\\
&\left((\frac{1}{2}-\hat{u}^{\alpha})(\frac{1}{2}-v^{\beta}),\frac{1}{2}+v^{\beta},(\frac{1}{2}+\hat{u}^{\alpha})(\frac{1}{2}-v^{\beta}) \right),\\
&\left((\frac{1}{2}+\hat{u}^{\alpha})(\frac{1}{2}-v^{\beta}), (\frac{1}{2}-\hat{u}^{\alpha})(\frac{1}{2}-v^{\beta}),\frac{1}{2}+v^{\beta}\right) \bigg \},
\end{align*}
where
$\{\hat{u}^{\alpha}$, $\alpha =1,\cdots, N\}$ and $\{v^{\beta}$, $\beta=1,\cdots, k+1\}$ are the Gauss-Lobatto quadrature points and Gauss quadrature points on $[-\frac{1}{2},\frac{1}{2}]$ respectively and $N$ is chosen such that $2N-3\geq k$, where $k$ is the degree of approximation polynomials. 

We now show that this set can be used as a test set for the IRP DG schemes  in Theorem \ref{cfl2d}.  Based on this test set, a decomposition of the cell average that is of form (\ref{avgdec}) satisfying (\ref{qwts}) was given in \cite[(3.5)]{ZXS12}:
\begin{align*}
\bar{\mathbf{w}}^n_K=\sum ^P_{\alpha =1}c_{\alpha}\mathbf{w}^{\alpha}_K + \sum ^3_{i=1}\sum ^{k+1}_{\beta =1}\frac{2}{3}w_{\beta}\hat{w}_1\mathbf{w}^{i,\beta}_K,
\end{align*} 
where $P=3(N-2)(k+1)$,  $\mathbf{w}^{\alpha}_K$ are approximations to solution values at quadrature points in the interior of $K$, $c_{\alpha}$ are the corresponding weights, $\mathbf{w}^{i,\beta}$ are approximations at $\beta-$th quadrature point on $e^i_K$ from cell $K$, $w_{\beta}$ are weights of Gauss quadrature points and $\hat{w}_1$ is the weight of first Gauss-Lobatto quadrature point. Then, according to Theorem \ref{cfl2d}, we find that the CFL condition for IRP DG schemes on triangular meshes is  
\begin{align}\label{tcfl}
\sigma\frac{\Delta t}{|K|}\leq \frac{2}{3|\partial K|}\hat{w}_1c_0,
\end{align}
where $\sigma=\|\sqrt{u^2+v^2}+c\|_{\infty}$ and $c_0$ is $\frac{1}{2}$ or $1$ depending on the used IRP flux. Here we have used the fact that the eigenvalues for the Jacobian matrix $\partial_{w}(F(w)\cdot \nu)$
 are $\{\mathbf{u}\cdot \nu-c, \mathbf{u}\cdot \nu, \mathbf{u}\cdot \nu+c, \mathbf{u}\cdot \nu\}$, where $\mathbf{u}=(u,v)$,  and $\nu$ is the unit vector, see \cite{Rohde}.
\end{enumerate}

With the test sets and CFL conditions given above, the IRP limiter (\ref{IRlimiter}) with (\ref{thetaEuler}) can be applied, where now $x\in K\subset \mathbb{R}^2$.

\section{Numerical tests} 
In this section, we present numerical examples to test the performance of IRP DG schemes presented in previous sections. Unless it is stated specifically, the IRP flux used is the local Lax-Friedrichs flux. $\gamma =1.4$ is taken for all of the examples.

The semi-discrete DG scheme is a closed ODE system 
$$
\frac{d}{dt}\mathbf{W}=L(\mathbf{W}),
$$
where $\mathbf{W}$ consists of the unknown coefficients of the spatial basis, and  $L$ is the corresponding spatial operator. 

For the time discretization we use the third order SSP Runge-Kutta (RK3) method introduced in \cite{SO88}:
\begin{align}\label{RK3}
\nonumber \mathbf{W}^{(1)}=&\mathbf{W}^n+\Delta tL(\mathbf{W}^n),\\
\nonumber \mathbf{W}^{(2)}=&\frac{3}{4}\mathbf{W}^n+\frac{1}{4}\mathbf{W}^{(1)}+\frac{1}{4}\Delta tL(\mathbf{W}^{(1)}),\\
\mathbf{W}^{n+1}=&\frac{1}{3}\mathbf{W}^n+\frac{2}{3}\mathbf{W}^{(2)}+\frac{2}{3}\Delta tL(\mathbf{W}^{(2)}).
\end{align}
We apply the IRP limiter at each time stage in the RK3 method with $\epsilon =10^{-13}$ for all of the examples. Notice that (\ref{RK3}) is a linear convex combination of Euler forward method, therefore the invariant region is preserved by the full scheme if it is preserved at each time stage.  
Unless specified otherwise, in 1D problems, we use the time step $\Delta t=\frac{\Delta x}{4\sigma}$ for $P^1$-DG scheme and $\Delta t=\frac{\Delta x}{12\sigma}$ for $P^2$-DG scheme, where $\sigma$ is the global maximum wave speed. In 2D problems, we use the time step $\Delta t=\frac{1}{4\eta}$ for $P^1$-DG scheme and $\Delta t=\frac{1}{12\eta}$, where $\eta =\frac{\sigma 1}{\Delta x}+\frac{\sigma 2}{\Delta y}$,  
$\sigma _1$ and $\sigma _2$ are global maximum wave speeds along $x$ direction and $y$ direction,  respectively.\\

\noindent{\bf Example 1.}  {\sl 1D accuracy tests}\\
We test the accuracy of the IRP DG scheme solving the one-dimensional compressible Euler equations with periodic boundary conditions. 
The initial condition is 
\begin{align*}
\rho _0(x)=1+0.5\sin(2\pi x), \quad u_0(x)=1, \quad p_0(x)=1
\end{align*}
in domain $[0,1]$.  The exact solution is 
$$
\rho (x,y,t)=1+0.5\sin(2\pi(x-t)), \;u(x,t)=1, \;p(x,t)=1.
$$ 
{The final time is taken as $T=0.1$.} Listed in Table \ref{P1a} and Table \ref{P2a} are  the errors and orders of accuracy of the DG method for density with and without
 the IRP limiter, respectively. 
 For the $P^1$-DG scheme, we observe the desired order of accuracy for our IRP DG method, justifying that the IRP limiter does not destroy the accuracy for smooth solutions. We also see that  the errors with or without the limiter are comparable.
While for $P^2$-DG scheme with IRP limiter, the third order of accuracy is observed before the mesh is refined to $N=512$. The reason for the order loss with finer meshes could be that the DG polynomials in the intermediate stages in the Runge-Kutta method are already not of the desired order, then the accuracy of modified polynomials using the IRP limiter is also affected. Such phenomenon has also been observed and discussed in \cite{ZS12}, in which the SSP multi-step time discretization is used to avoid such issue. We also observe that if we halve the time step for $P^2$-DG scheme with the IRP limiter, then the desired order of accuracy can be recovered. 

In order to see that the IRP limiter is indeed turned on, we present Table \ref{volP1} (for $P^1$-DG scheme) and Table \ref{volP2} (for $P^2$-DG scheme) showing some information when the IRP limiter is called for the first time in the code. Here $N$ is the number of meshes used in spatial discretization, $j$ denotes that the limiter is applied in the $j$-th cell, $q_{\max}$ is defined in (\ref{ElMinMax}) and $\theta$ is the limiter parameter given in (\ref{IRlimiter}). For the $P^3$-DG, numerical solutions remain in the invariant region, while no IRP limiter is called. \\
 
\begin{table}[htbp]
\centering
\begin{tabular}{|c|cccc||cccc|}
\hline
$P^1$ DG&\multicolumn{4}{|c||}{Without limiter}  &\multicolumn{4}{|c|}{With IRP limiter}
\\ \hline
N &$L^{\infty}$Error &Order &$L^1$Error & Order &$L^{\infty}$Error &Order &$L^1$Error & Order\\
\hline
16   &4.97E-03 &/     &1.79E-03  &/ &7.19E-03    &/     &1.87E-03 &/ \\
32   &1.27E-03 &1.98  &4.43E-04  &2.02 &2.17E-03 &1.73  &4.69E-04 &1.99  \\
64   &3.18E-04 &1.99  &1.10E-04  &2.00 &5.64E-04 &1.94  &1.14E-04 &2.04 \\
128  &7.96E-05 &2.00  &2.76E-05  &2.00 &1.49E-04 &1.92  &2.81E-05 &2.02 \\
256  &1.99E-05 &2.00  &6.89E-06  &2.00 &3.82E-05 &1.96  &7.01E-06 &2.00 \\
512  &4.99E-06 &2.00  &1.72E-06  &2.00 &1.01E-05 &1.91  &1.75E-06 &2.00 \\
1024  &1.25E-06 &2.00 &4.30E-07  &2.00 &2.67E-06 &1.93  &4.35E-07 &2.01 \\
\hline
\end{tabular}
\vspace{5pt}
\caption{Accuracy of density function in Example 1 using $P^1$-DG scheme.}
\label{P1a}
\end{table}

\begin{table}[htbp]
\centering
\begin{tabular}{|c|cccc||cccc|}
\hline
$P^2$ DG&\multicolumn{4}{|c||}{Without limiter}  &\multicolumn{4}{|c|}{With IRP limiter}
\\ \hline
N &$L^{\infty}$Error &Order &$L^1$Error & Order &$L^{\infty}$Error &Order &$L^1$Error & Order\\
\hline
16   &3.20E-04 &/     &1.34E-04 &/     &3.35E-04 &/     &1.46E-04 &/ \\
32   &4.35E-05 &2.88  &1.78E-05 &2.92  &4.35E-05 &2.94  &1.80E-05 &3.01  \\
64   &5.54E-06 &2.97  &2.25E-06 &2.98  &5.54E-06 &2.98  &2.28E-06 &2.98 \\
128  &6.94E-07 &3.00  &2.83E-07 &2.99  &6.94E-07 &3.00  &2.87E-07 &2.99 \\
256  &8.68E-08 &3.00  &3.54E-08 &3.00  &8.68E-08 &3.00  &3.62E-08 &2.99 \\
512  &1.08E-08 &3.00  &4.42E-09 &3.00  &2.02E-08 &2.10  &4.54E-09 &2.99 \\
1024 &1.35E-09 &3.00  &5.53E-10 &3.00  &5.71E-09 &1.83  &5.72E-10 &2.99 \\
\hline
\end{tabular}
\vspace{5pt}
\caption{Accuracy test for density function in Example 1 using $P^2$-DG scheme.}
\label{P2a}
\end{table}

\begin{table}[htbp]
\centering
\begin{tabular}{|c|c|c|c|c|}
\hline
N &j &$q_{\max}$ &$q(\mathbf{\bar{w}}_j)$ &$\theta$ \\
\hline
16 &4 &4.48E-03 &-2.29E-02 &0.8360\\
32 &8 &1.02E-03 &-5.75E-03 &0.8491 \\
64 &16 &2.49E-04 &-1.44E-03 &0.8525 \\
128 &32 &6.18E-05 &-3.60E-04 &0.8534 \\
256 &64 &1.54E-05 &-9.00E-05 &0.8536 \\
512 &128 &3.86E-06 &-2.25E-05 &0.8537 \\
1024 &256 &9.64E-07 &-5.62E-06 &0.8537 \\
\hline
\end{tabular}
\vspace{5pt}
\caption{Some parameter values when the limiter is first called in the $P^1$-DG scheme solving Example 1.}
\label{volP1}
\end{table}

\begin{table}[htbp]
\centering
\begin{tabular}{|c|c|c|c|c|}
\hline
N &j &$q_{\max}$ &$q(\mathbf{\bar{w}}_j)$ &$\theta$ \\
\hline
16 &5 &6.25E-05 &-1.61E-02 &0.996141\\
32 &9 &3.91E-06 &-4.07E-03 &0.999040\\
64 &17 &2.44E-07 &-1.02E-03 &0.999760 \\
128 &33 &1.53E-08 &-2.55E-04 &0.999940 \\
256 &65 &9.54E-10 &-6.37E-05 &0.999985 \\
512 &129 &5.96E-11 &-1.59E-05 &0.999996 \\
1024 &257 &3.73E-12 &-3.98E-06 &0.999999 \\
\hline
\end{tabular}
\vspace{5pt}
\caption{Some parameter values when the limiter is first called in the $P^2$-DG scheme solving Example 1.}
\label{volP2}
\end{table}


%
\noindent{\bf Example 2.} {\sl 2D accuracy tests}\\
We test the accuracy of the IRP DG scheme on a low density problem in two dimensional case. The initial condition is 
\begin{align*}
\rho _0(x,y)=1+0.99\sin(x+y), \quad u_0(x,y)=1, \quad v_0(x,y)=1, \quad p_0(x,y)=1
\end{align*}
in domain $[0,2\pi]$ with periodic boundary conditions. The exact solution is 
$$
\rho (x,y,t)=1+0.99\sin(x+y-2t), u(x,y,t)=1, v(x,y,t)=1, p(x,y,t)=1.
$$
{The final time is $T=0.1$.} From Table \ref{P1b} and Table \ref{P2b} we can see close to $k+1$-th order of accuracy of the IRP DG scheme with polynomials of degree $k$. 
Listed In Table \ref{vol2DP1} (for $P^1$-DG scheme) and Table \ref{vol2DP2} (for $P^2$-DG scheme) are  values of $\theta$ and related indicators when the IRP limiter is called  in the code for the first time. $(i,j)$ denotes that the limiter is used in the $(i,j)$-th cell.\\

\begin{table}[htbp]
\centering
\begin{tabular}{|c|cccc||cccc|}
\hline
$P^1$ DG&\multicolumn{4}{|c||}{Without limiter}  &\multicolumn{4}{|c|}{With IRP limiter}
\\ \hline
$N\times N$ &$L^{\infty}$Error &Order &$L^1$Error & Order &$L^{\infty}$Error &Order &$L^1$Error & Order\\
\hline
$32\times 32$    &8.96E-03  &/    &2.12E-03   &/     &1.25E-02   &/    &2.45E-03  &/ \\
$64\times 64$    &2.18E-03  &2.04 &5.09E-04   &2.06  &2.88E-03   &2.11 &5.60E-04  &2.13 \\
$128\times 128$   &4.96E-04  &2.14 &1.24E-04   &2.04  &7.43E-04   &1.95 &1.32E-04  &2.09 \\
$256\times 256$   &1.23E-04  &2.02 &3.05E-05   &2.02  &1.86E-04   &1.99 &3.17E-05  &2.06 \\
$512\times 512$   &3.07E-05  &2.00 &7.58E-06   &2.01  &4.86E-05   &1.94 &7.76E-06  &2.03 \\
\hline
\end{tabular}
\vspace{5pt}
\caption{Accuracy test for density function in Example 2 using $P^1$-DG scheme.}
\label{P1b}
\end{table}

\begin{table}[htbp]
\centering
\begin{tabular}{|c|cccc||cccc|}
\hline
$P^2$ DG&\multicolumn{4}{|c||}{Without limiter}  &\multicolumn{4}{|c|}{With IRP limiter}
\\ \hline
N &$L^{\infty}$Error &Order &$L^1$Error & Order &$L^{\infty}$Error &Order &$L^1$Error & Order\\
\hline
16    &2.84E-03   &/    &6.37E-04  &/     &3.66E-03  &/    &7.00E-04  &/ \\
32    &4.04E-04   &2.82 &7.92E-05  &3.01  &4.04E-04  &3.18 &8.59E-05  &3.03 \\
64    &5.11E-05   &2.98 &9.83E-06  &3.01  &5.11E-05  &2.98 &1.00E-05  &3.10 \\
128   &6.39E-06   &3.00 &1.22E-06  &3.01  &6.39E-06  &2.99 &1.23E-06  &3.03 \\
256   &8.06E-07   &2.99 &1.51E-07  &3.01  &8.06E-07  &2.97 &1.52E-07  &3.02 \\
\hline
\end{tabular}
\vspace{5pt}
\caption{Accuracy test for density function in Example 2 using $P^2$-DG scheme.}
\label{P2b}
\end{table}

\begin{table}[htbp]
\centering
\begin{tabular}{|c|c|c|c|c|}
\hline
$N\times N$ &(i,j) &$q_{\max}$ &$q(\mathbf{\bar{w}}_j)$ &$\theta$ \\
\hline
$32\times 32$ &(7,1) &1.95E-02 &-3.08E-02 &0.6119\\
$64\times 64$ &(15,1) &4.93E-03 &-7.77E-03 &0.6120\\
$128\times 128$ &(31,1) &1.23E-03 &-1.95E-03 &0.6120\\
$256\times 256$ &(63,1) &3.09E-04 &-4.87E-04 &0.6120\\
$512\times 512$ &(127,1) &7.72E-05 &-1.22E-04 &0.6120\\
\hline
\end{tabular}
\vspace{5pt}
\caption{Some parameter values when the limiter is first called in the $P^1$-DG scheme solving Example 2. }
\label{vol2DP1}
\end{table}

\begin{table}[htbp]
\centering
\begin{tabular}{|c|c|c|c|c|}
\hline
$N\times N$ &(i,j) &$q_{\max}$ &$q(\mathbf{\bar{w}}_j)$ &$\theta$ \\
\hline
$16\times 16$ &(3,1) &1.98E-03 &-0.12 &0.9836\\
$32\times 32$ &(7,1) &8.26E-05 &-3.08E-02 &0.9973\\
$64\times 64$ &(17,1) &1.52E-06 &-7.60E-03 &0.9998\\
$128\times 128$ &(33,1) &2.88E-07 &-1.84E-03 &0.9998\\
$256\times 256$ &(65,1) &2.69E-08 &-4.55E-04 &0.9999\\
\hline
\end{tabular}
\vspace{5pt}
\caption{Some parameter values when the limiter is first called in the $P^2$-DG scheme solving Example 2. }
\label{vol2DP2}
\end{table}

In the following examples, we test the IRP DG  schemes solving 1D and 2D Riemann problems, respectively.  We compare the results between those with the IRP limiter  (\ref{IRlimiter}) and those with only positivity-preserving limiter; that is, using $\theta = \min\{1,\theta _1,\theta _2\}$, where $\theta _1$ and $\theta _2$ are defined as in (\ref{thetaEuler}). \\


\noindent{\bf Example 3.}  {\sl 1D Sod tube problem}\\
Consider the Sod initial data:
\begin{align*}
(\rho, m, E)=
\begin{cases}
(1, 0, 2.5), \quad &x<0,\\
(0.125, 0, 0.25), \quad &x\geq 0.
\end{cases}
\end{align*}
The exact solution, which can be obtained by using the formula given in \cite[section 14.11,14.12]{FVLeVeque}, consists of a composite wave, that is, a rarefaction wave followed by a contact discontinuity and then by a shock.   The numerical solution obtained from the $P^2$-DG scheme on 200 cells at final time $ T=0.16$ is displayed in Figure \ref{fig:Sod}, from which we can see that the IRP limiter helps to reduce the oscillations near the interface between the rarefaction waves and the contact discontinuity. It also helps to damp the overshoots and the undershoots. For example, with positivity-preserving limiter alone, the maximum and minimum of velocity solution is $1.004$ and $-0.177$,  respectively; while by using the IRP limiter, the maximum and minimum of velocity become $0.998$ and $-0.171$,  respectively.\\

\begin{figure}[htbp]
\centering
\includegraphics[height = 1.8in]{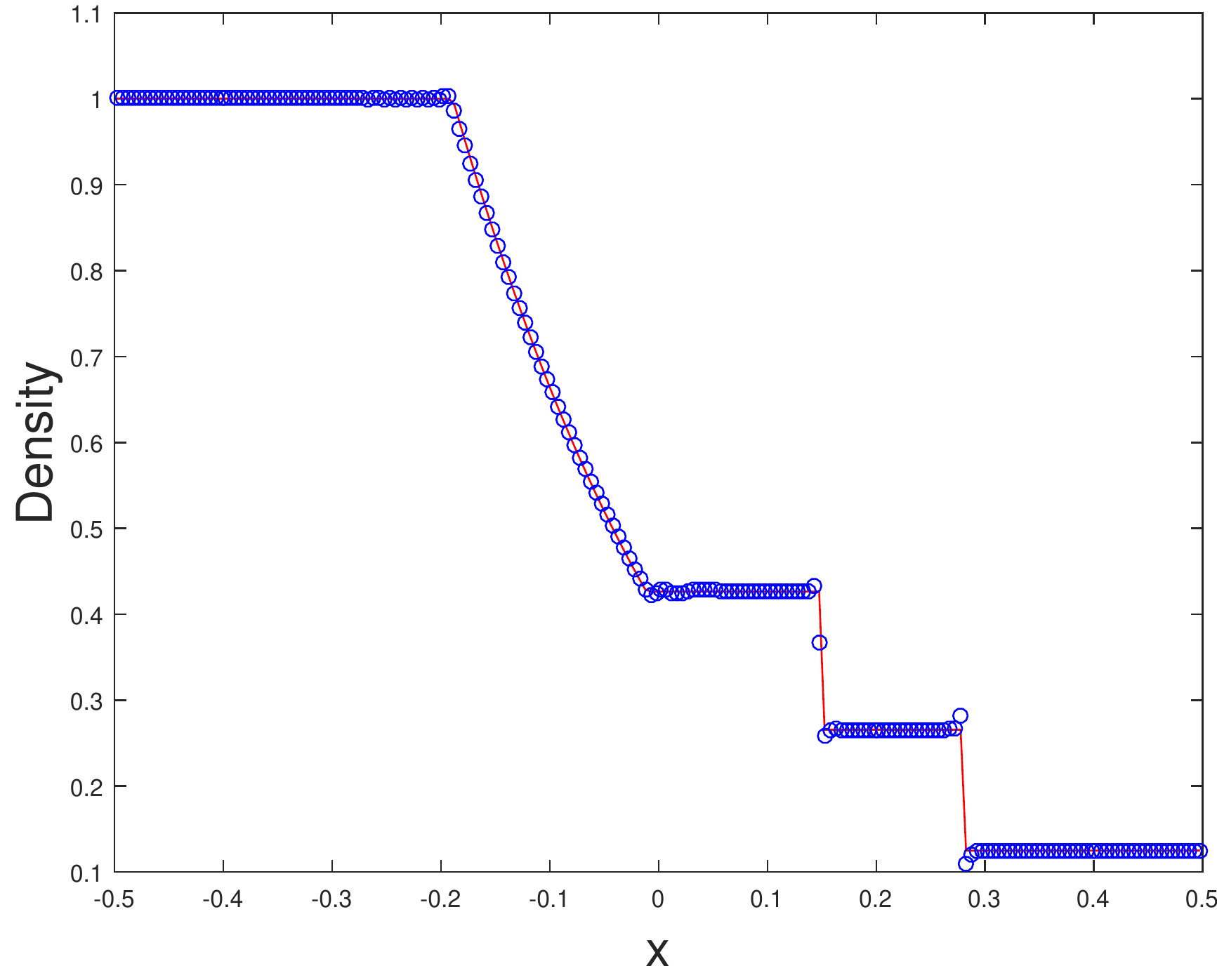}
\includegraphics[height = 1.8in]{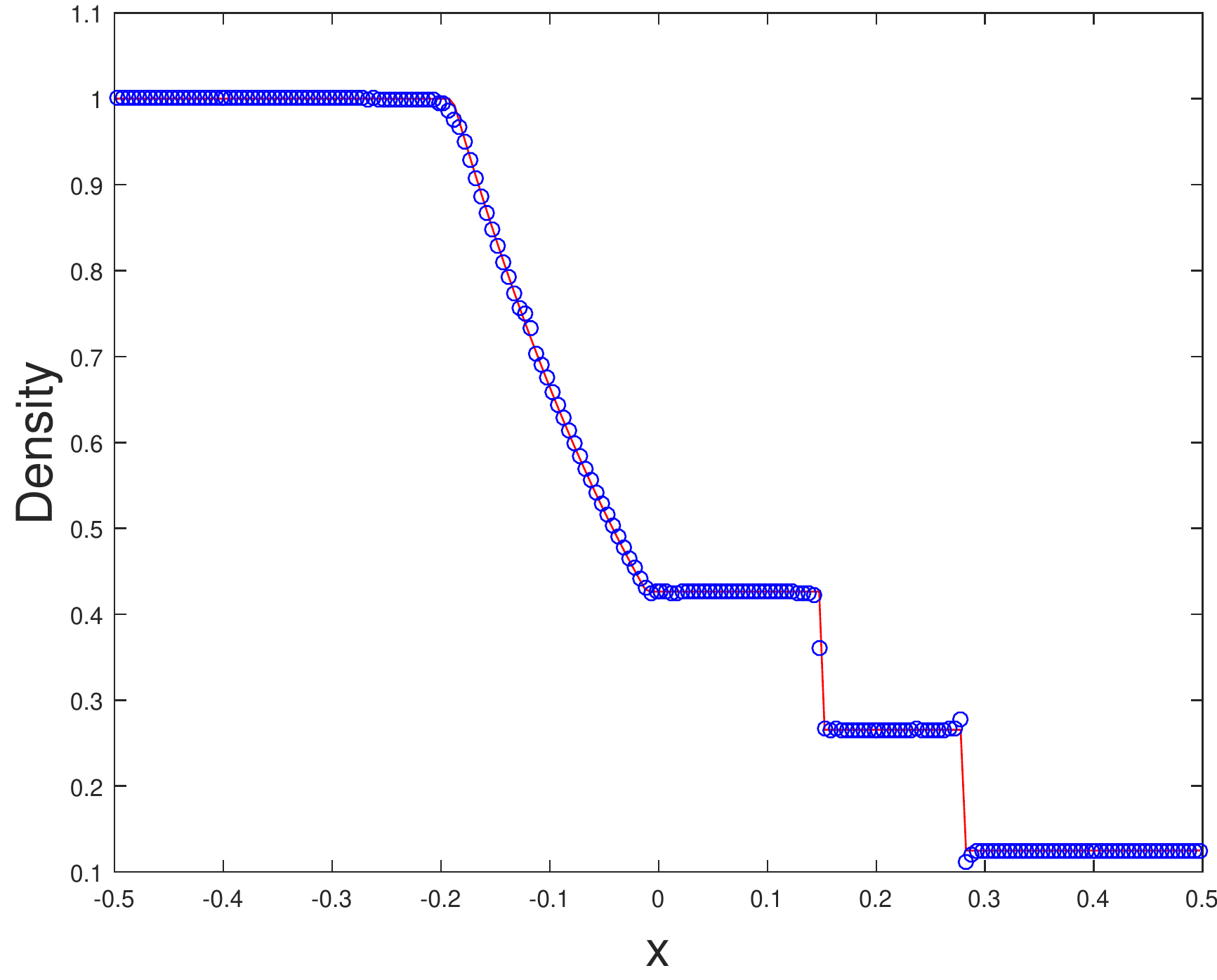}\\
\includegraphics[height = 1.8in]{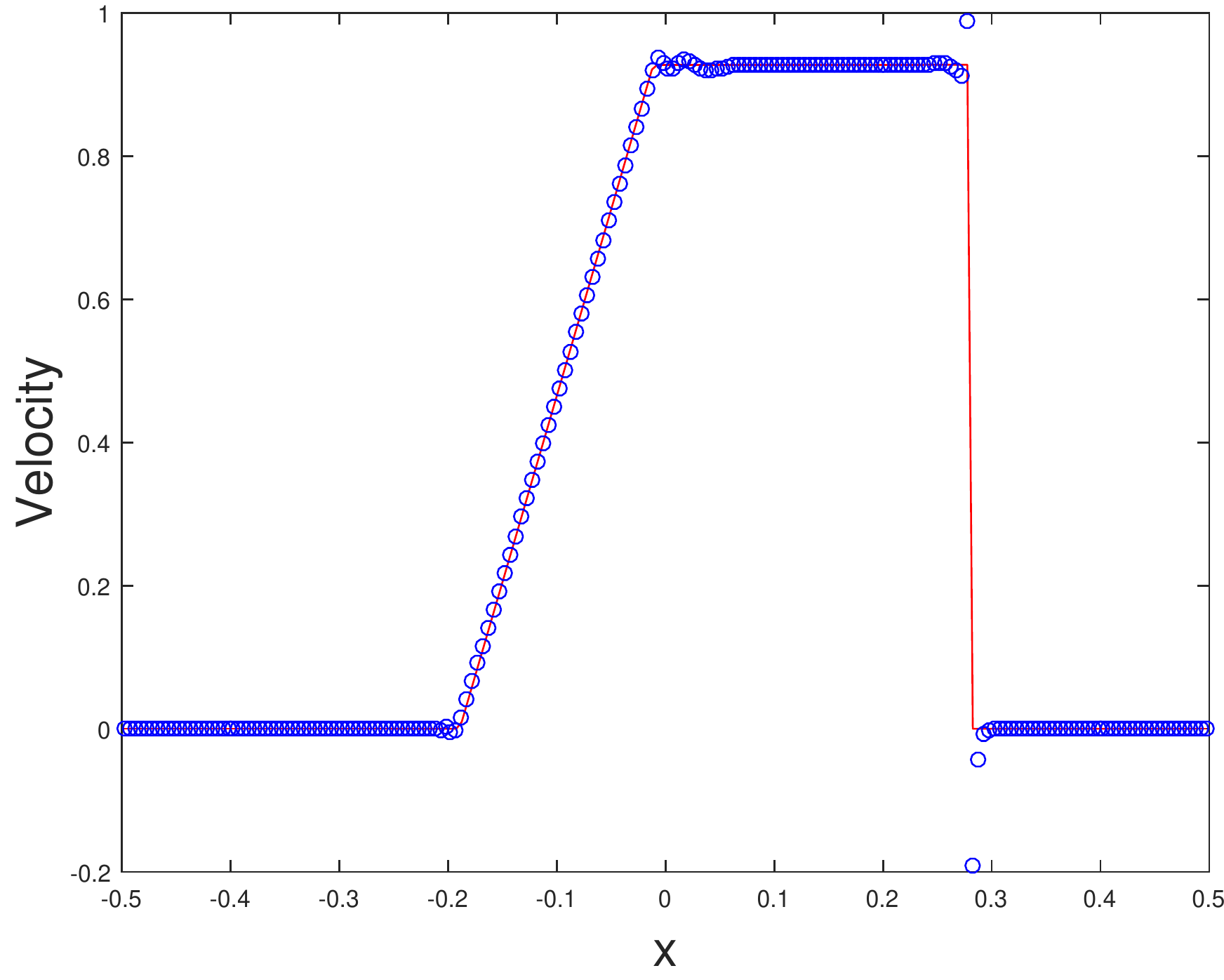}
\includegraphics[height = 1.8in]{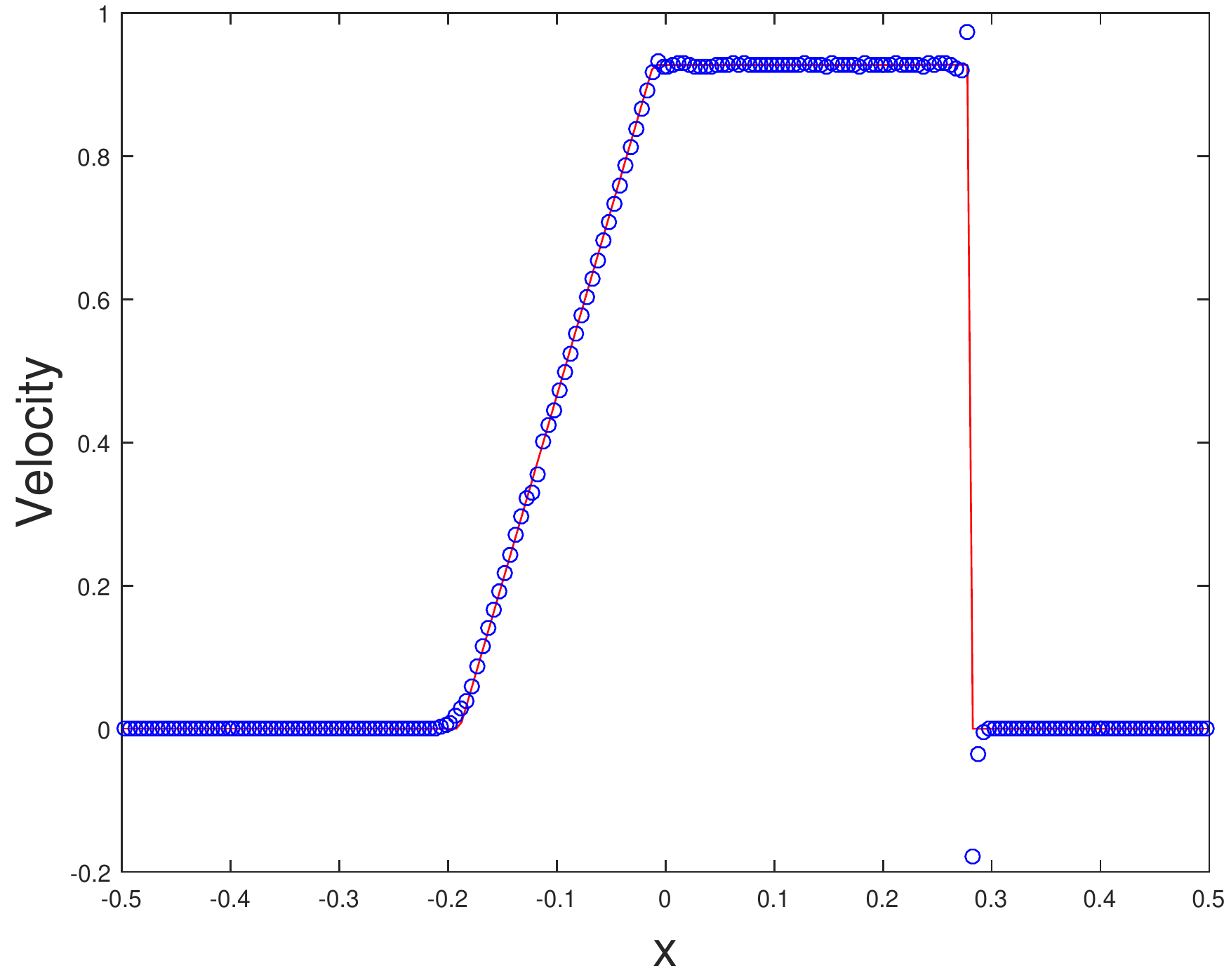}\\
\includegraphics[height = 1.8in]{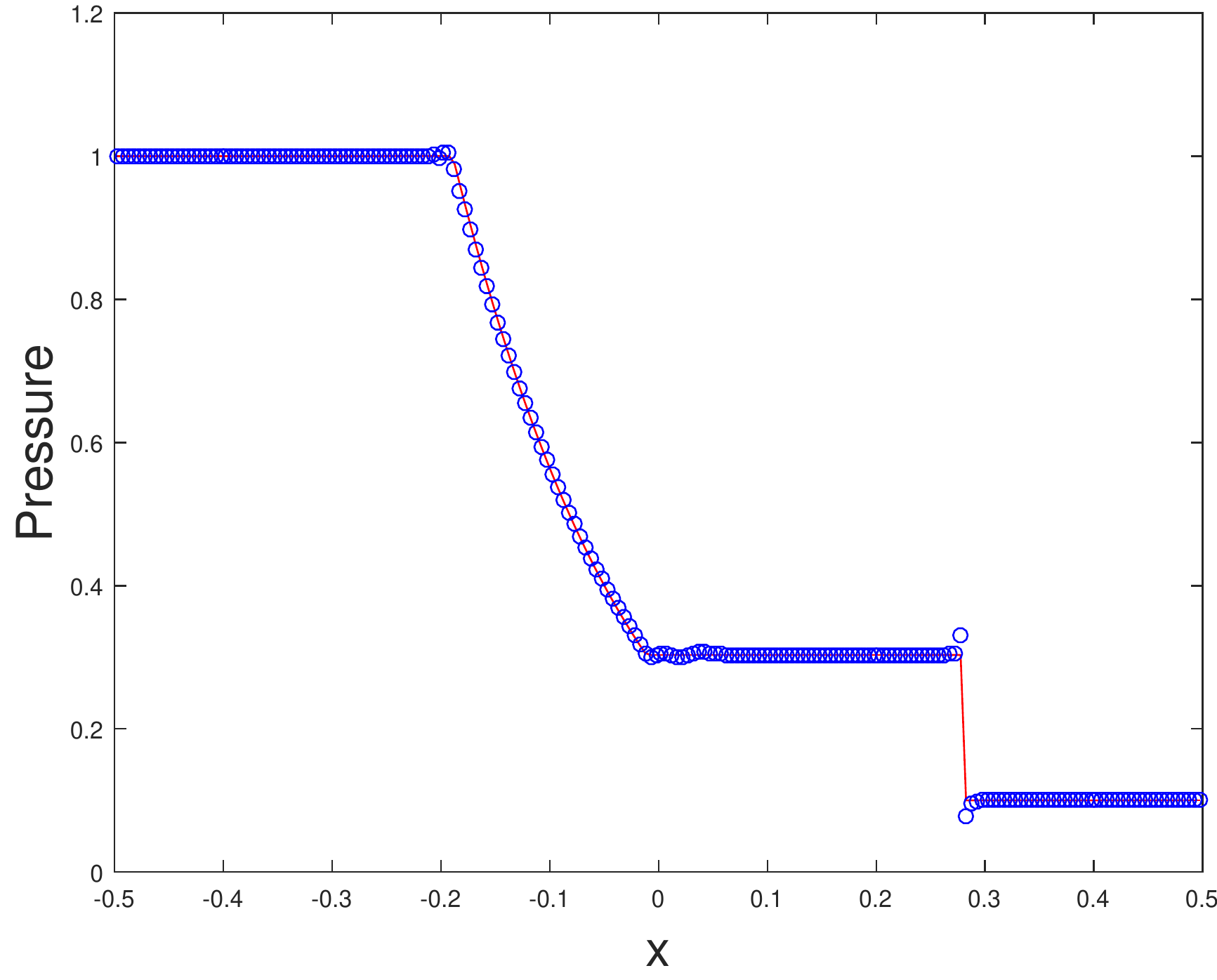}
\includegraphics[height = 1.8in]{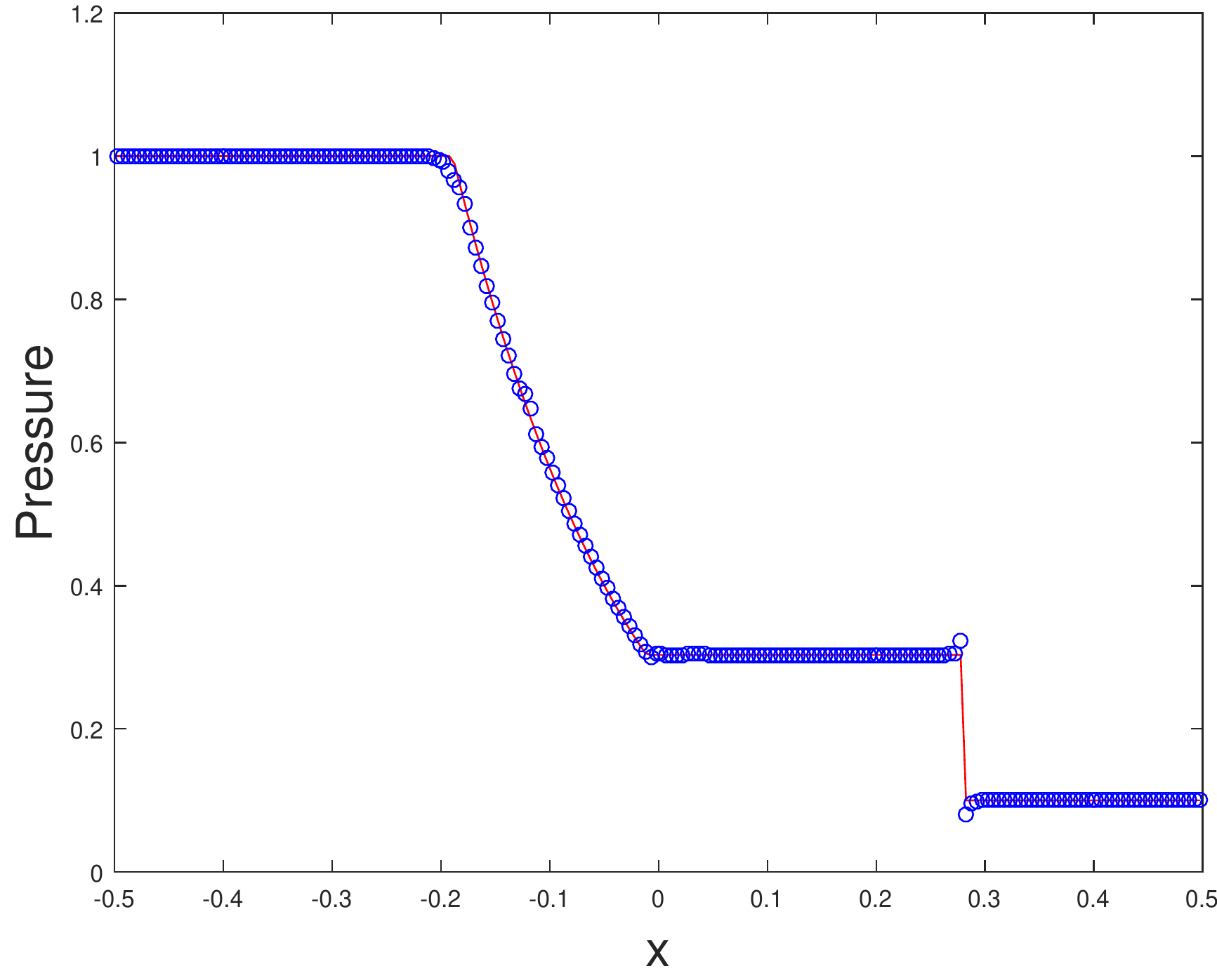}
\caption{Sod shock tube problem. Exact solution (solid line) vs numerical solution (dots); Left: With positive-preserving limiter; Right: With invariant-region-preserving limiter}
\label{fig:Sod}
\end{figure}  

\noindent{\bf Example 4.}{\sl 1D double rarefaction problem}\\
We consider the one dimensional Riemann problem with initial condition
\begin{align*}
(\rho, u, p)=
\begin{cases}
(1, -12, 1), \quad &x<0,\\
(1, 12, 1), \quad &x\geq 0.
\end{cases}
\end{align*}  
The exact solution consists of two double rarefaction waves moving in opposite directions, which results in the creation of a vacuum in the center of the domain. 
For this problem, we use the global Lax-Friedrich flux and a reduced time step in order to avoid the blow-ups in the computation. The numerical solution obtained from the $P^2$-DG scheme on 400 cells is displayed at the final time $T=0.3$ in Figure \ref{fig:DF}, from which we see that both numerical schemes capture the vacuum region well. Moreover, the IRP DG scheme can help to damp the 
{overshoots} near the top of the rarefactions.  Although we use  $\Delta t=\frac{\Delta x}{20\sigma}$ to obtain results in Figrue \ref{fig:DF}, it's been noticed that the time step required to obtain a reasonable solution by the IRP DG scheme ($\Delta t=\frac{\Delta x}{14\sigma}$) is less restricted than the one by the scheme using only positivity-preserving limiter ($\Delta t=\frac{\Delta x}{19\sigma}$).\\

\begin{figure}[htbp]
\centering
\includegraphics[height = 1.8in]{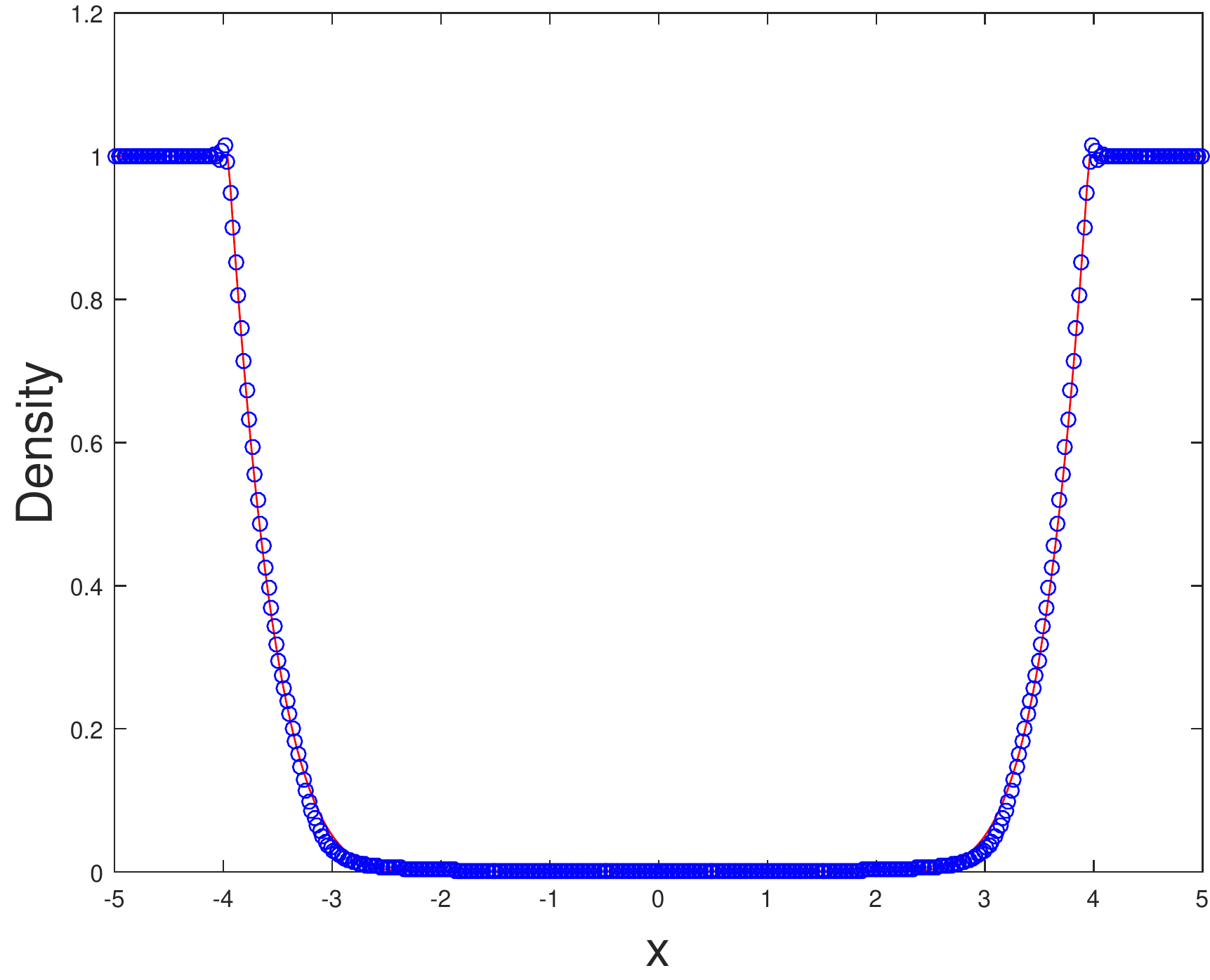}
\includegraphics[height = 1.8in]{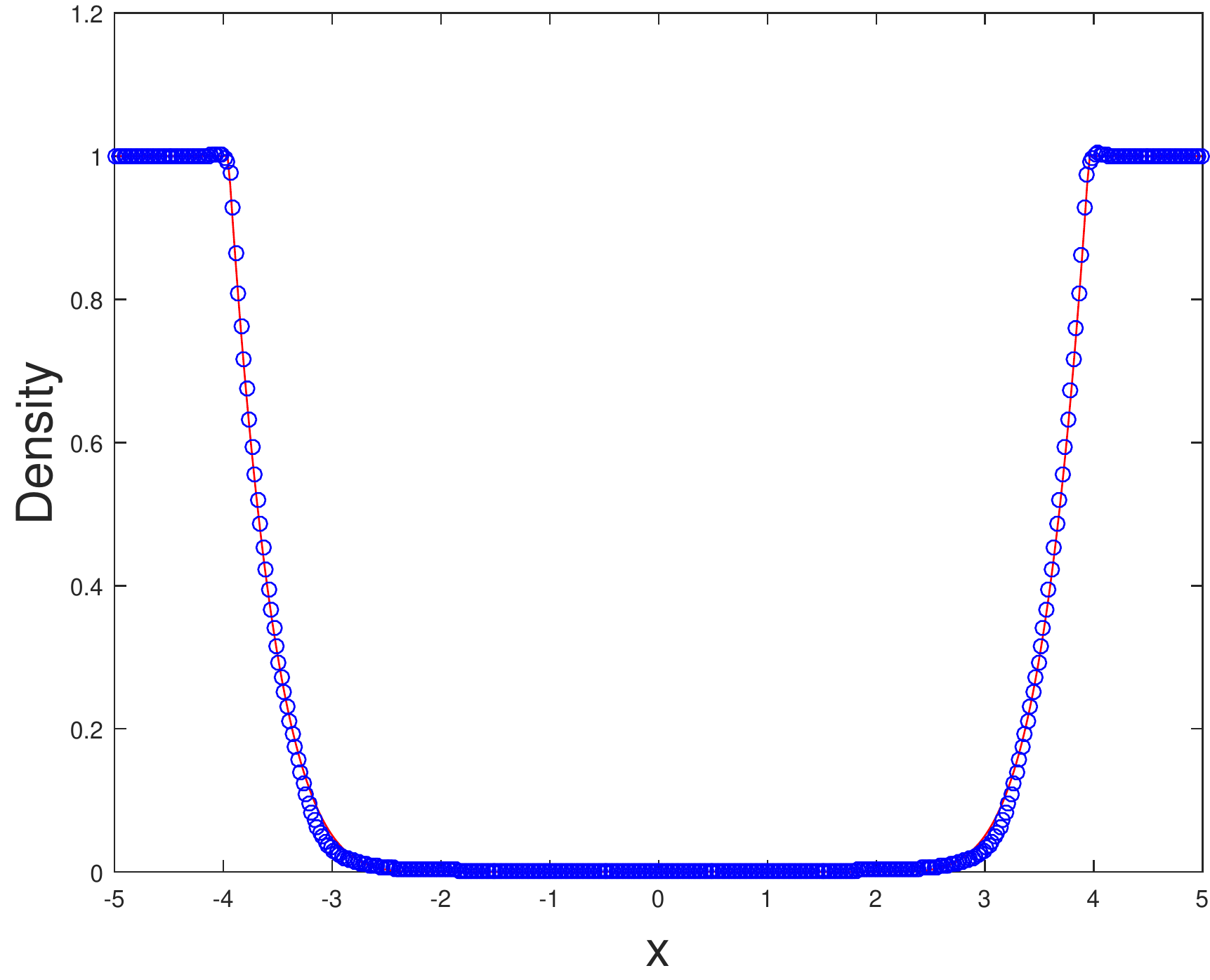}\\
\includegraphics[height = 1.8in]{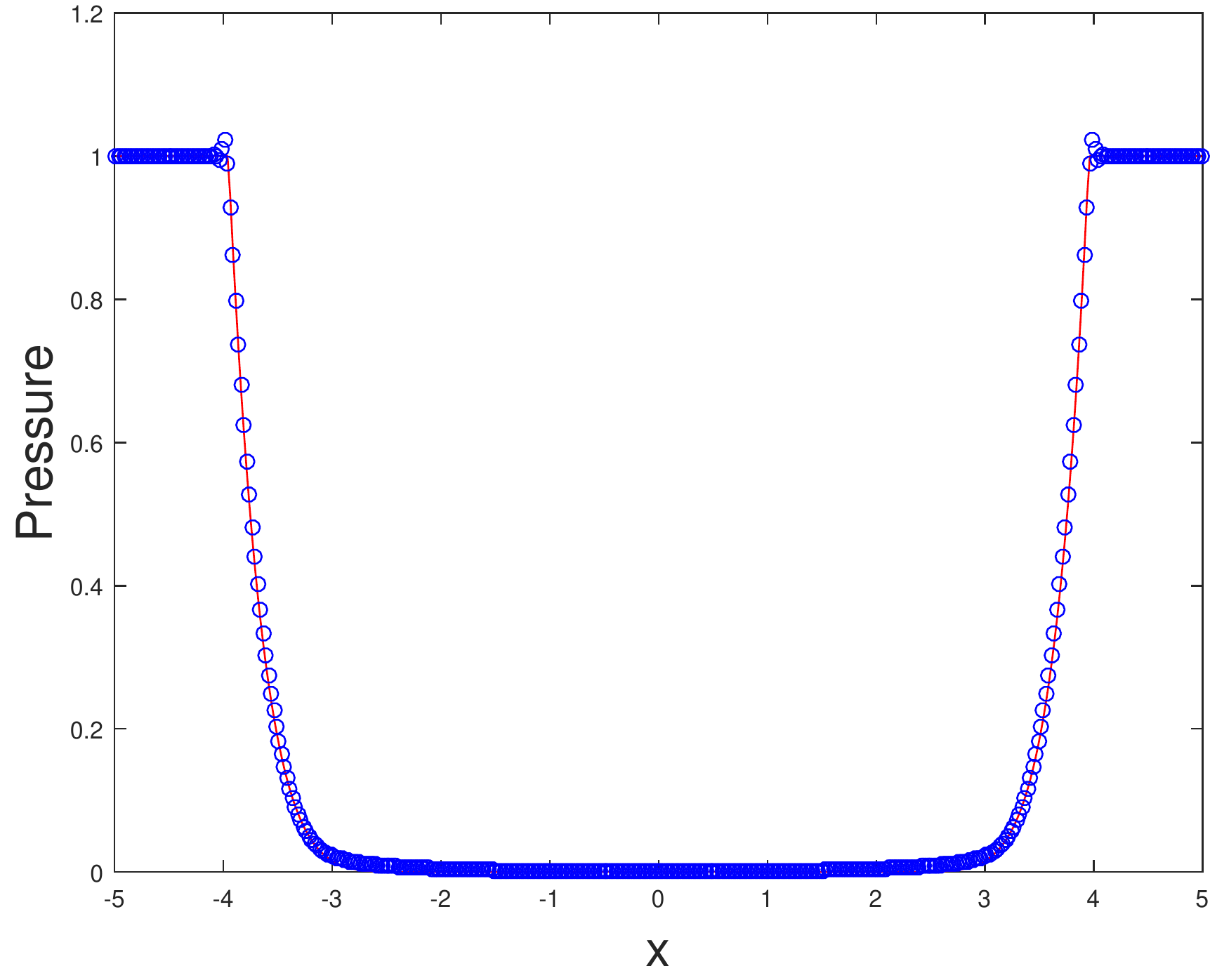}
\includegraphics[height = 1.8in]{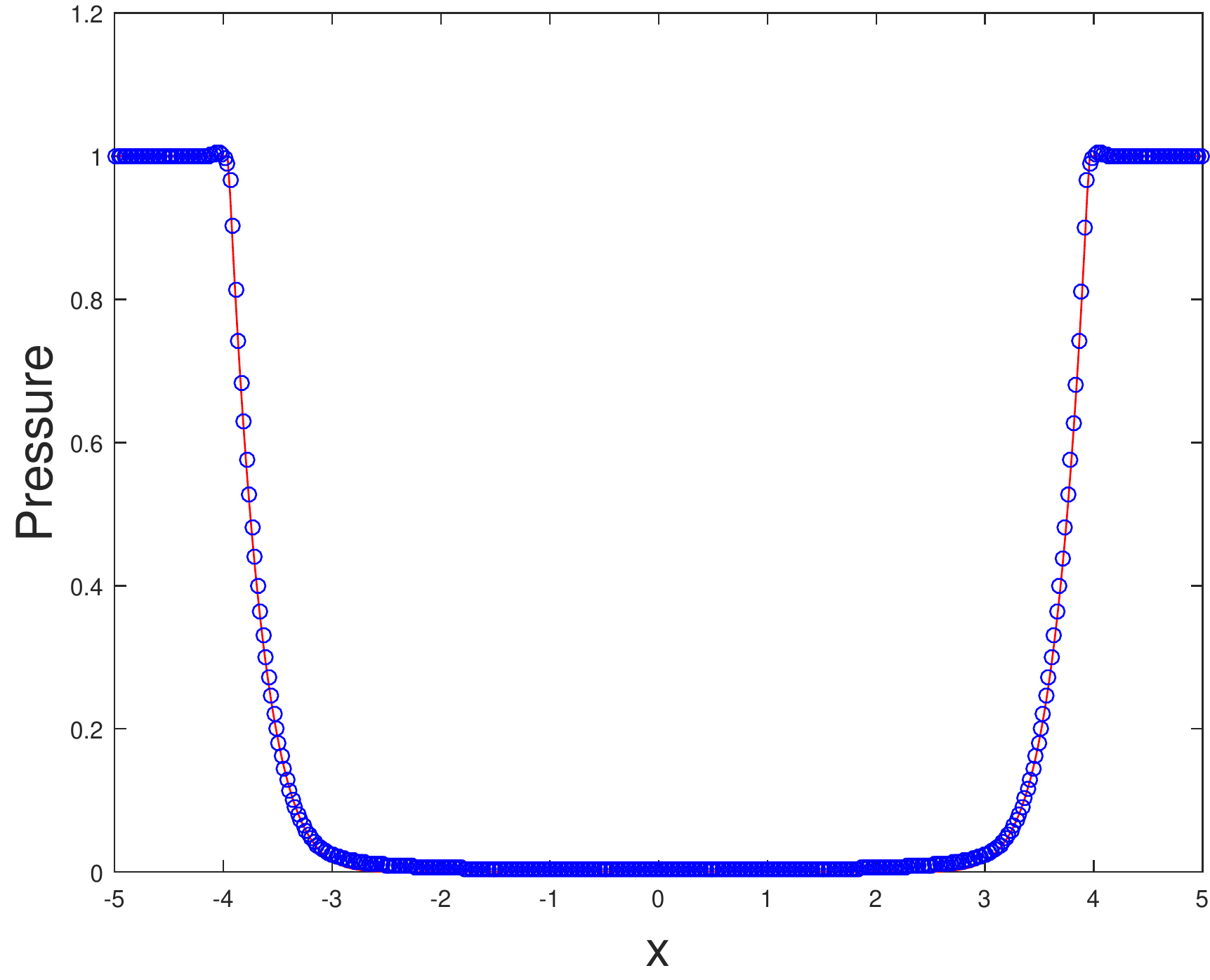}\\
\caption{1D double rarefaction problem. Exact solution (solid line) vs numerical solution (dots); Left: with positivity-preserving limiter; Right: with IRP limiter}
\label{fig:DF}
\end{figure}  

\noindent{\bf Example 5.}{\sl  2D Riemann problem}\\
For two dimensional Euler equations, there are nineteen configurations of Riemann solutions that have been studied in \cite{LL,ZCY}. In this example, we test only two configurations. 
 
The first configuration (Configuration 2 in \cite{LL}) has the initial condition as
\begin{align*}
(\rho, u, v, p)=
\begin{cases}
(1,0,0,1), \quad &(x,y)\in (0.5,1)\times (0.5,1)\\
(0.5197,-0.7259,0,0.4), \quad &(x,y)\in (0,0.5)\times (0.5,1)\\
(1,-0.7259,-0.7259,1), \quad &(x,y)\in (0,0.5)\times (0,0.5)\\
(0.5197,0,-0.7259,0.4), \quad &(x,y)\in (0.5,1)\times (0,0.5)
\end{cases}
.
\end{align*}
The solution consists of four rarefaction waves. From the contour plots in Figure \ref{fig:C2}, we can see the IRP limiter helps to make the solution smoother. 

The second configuration (Configuration 6 in \cite{LL}) has the initial condition as
\begin{align*}
(\rho, u, v, p)=
\begin{cases}
(1,0.75,-0.5,1), \quad &(x,y)\in (0.5,1)\times (0.5,1)\\
(2,0.75,0.5,1), \quad &(x,y)\in (0,0.5)\times (0.5,1)\\
(1,-0.75,0.5,1), \quad &(x,y)\in (0,0.5)\times (0,0.5)\\
(3,-0.75,-0.5,1), \quad &(x,y)\in (0.5,1)\times (0,0.5)
\end{cases}
.
\end{align*}
The solution consists of four two-dimensional slip lines. See Figure \ref{fig:C6}. We zoom in the plot near {the lower left and lower right interface between two constant states} and observe that the IRP limiter helps to damp some of the oscillations. See Figure \ref{fig:zmC6}.

\begin{figure}[htbp]
\centering
\includegraphics[height = 1.8in]{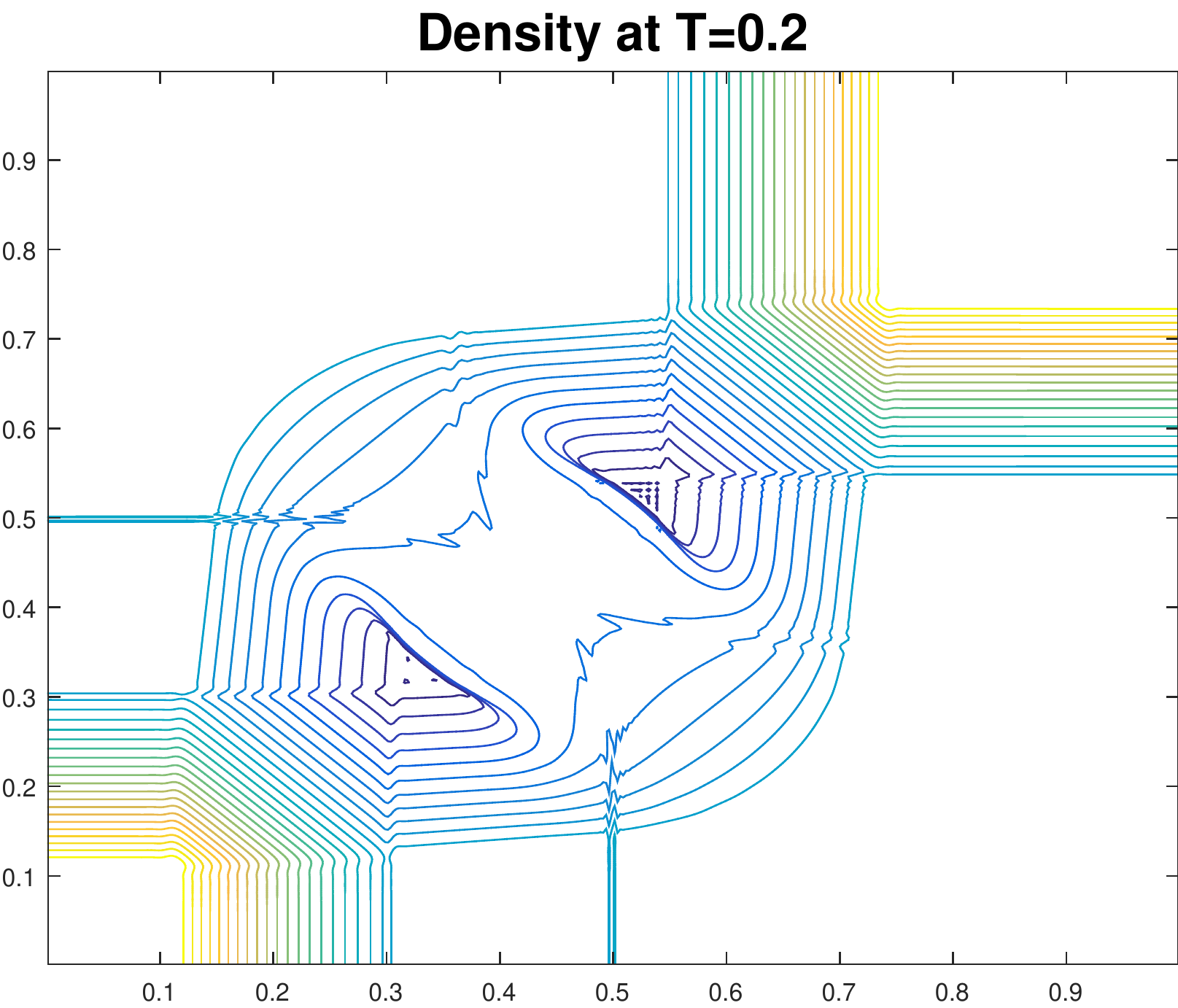}
\includegraphics[height = 1.8in]{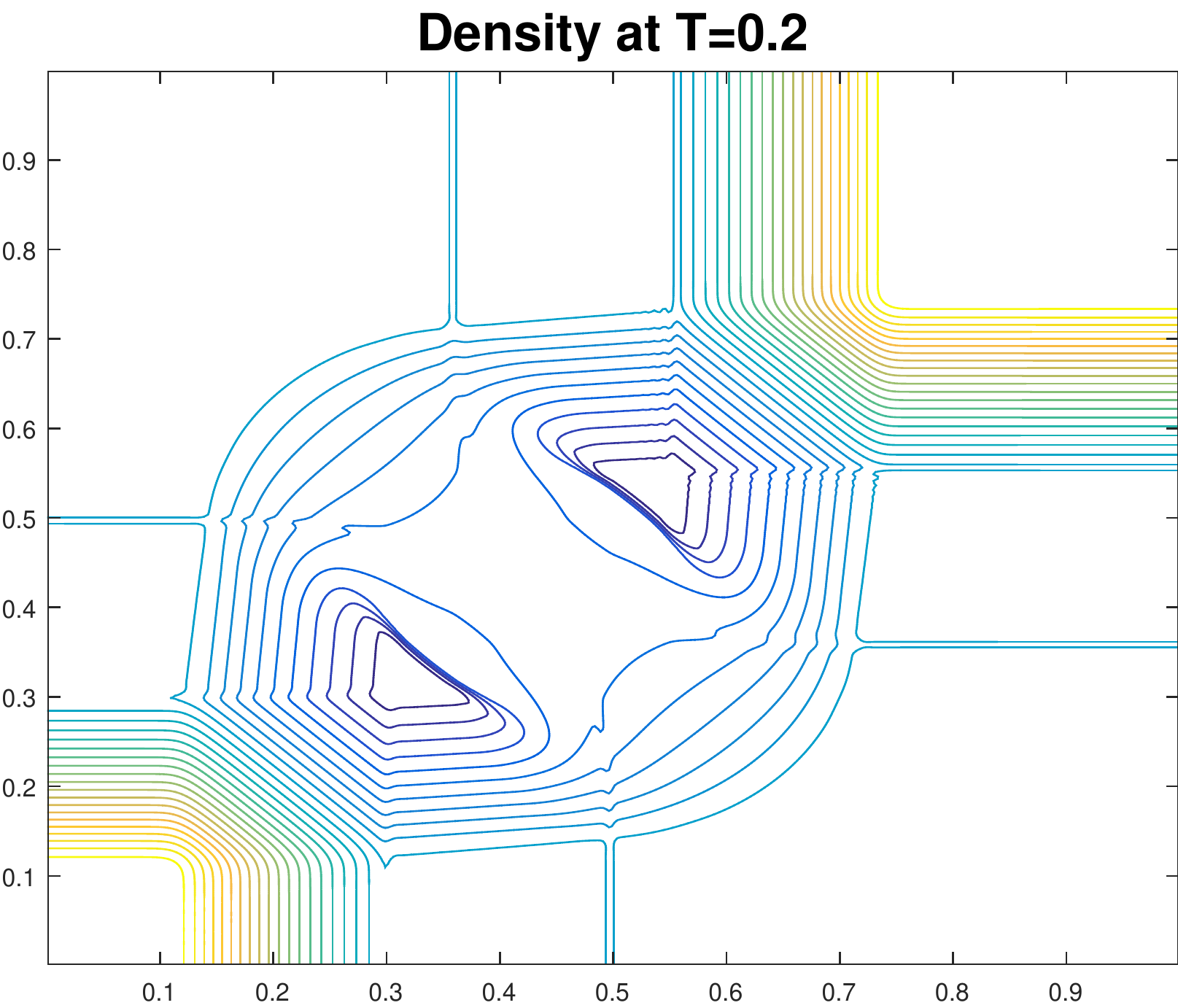}
\caption{Configuration 2. Contour plot of numerical solution of density with 30 contour levels. Left: with positive-preserving limiter; Right: with invariant-region-preserving limiter}
\label{fig:C2}
\end{figure} 

\begin{figure}[htbp]
\centering
\includegraphics[height = 1.8in]{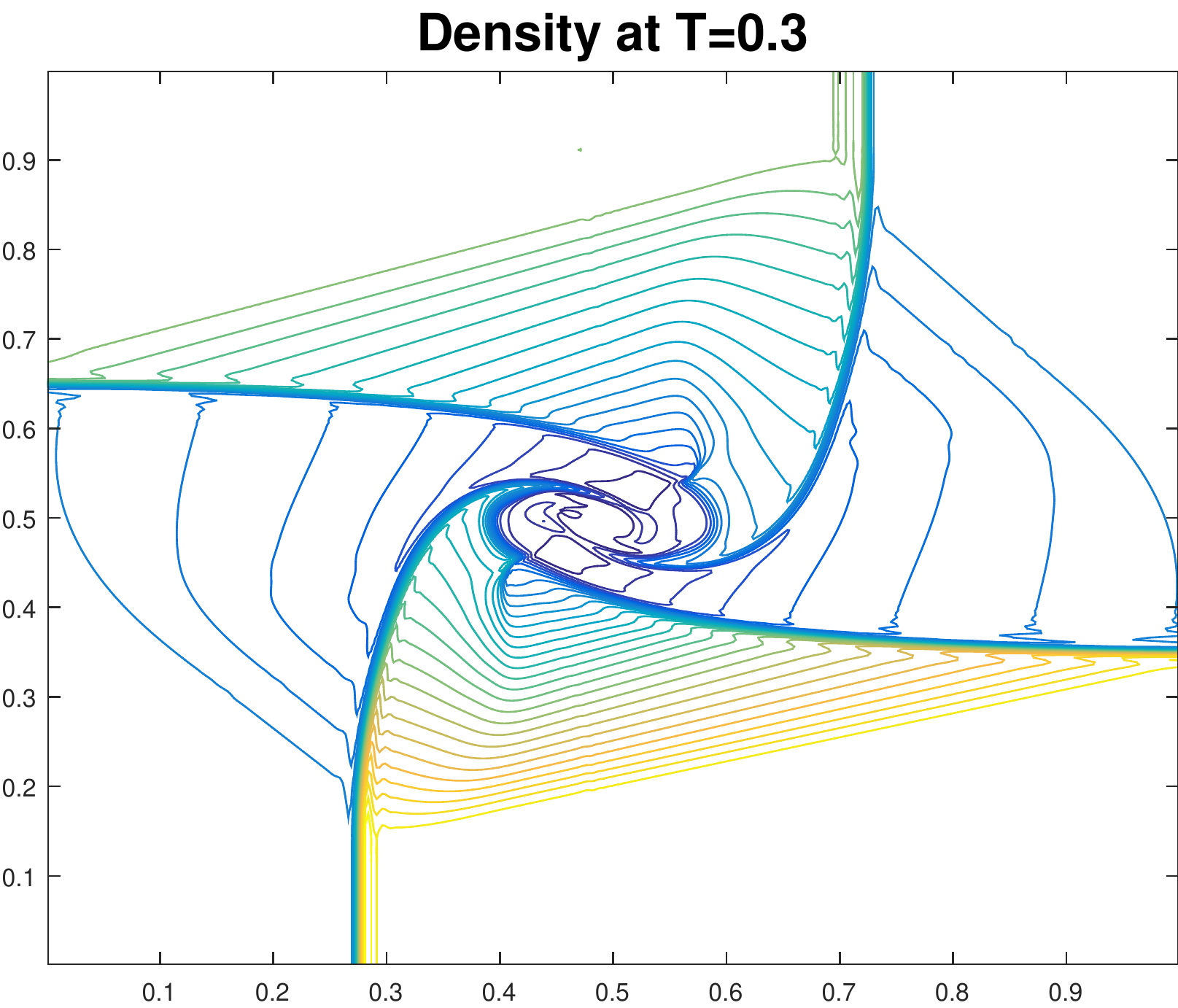}
\includegraphics[height = 1.8in]{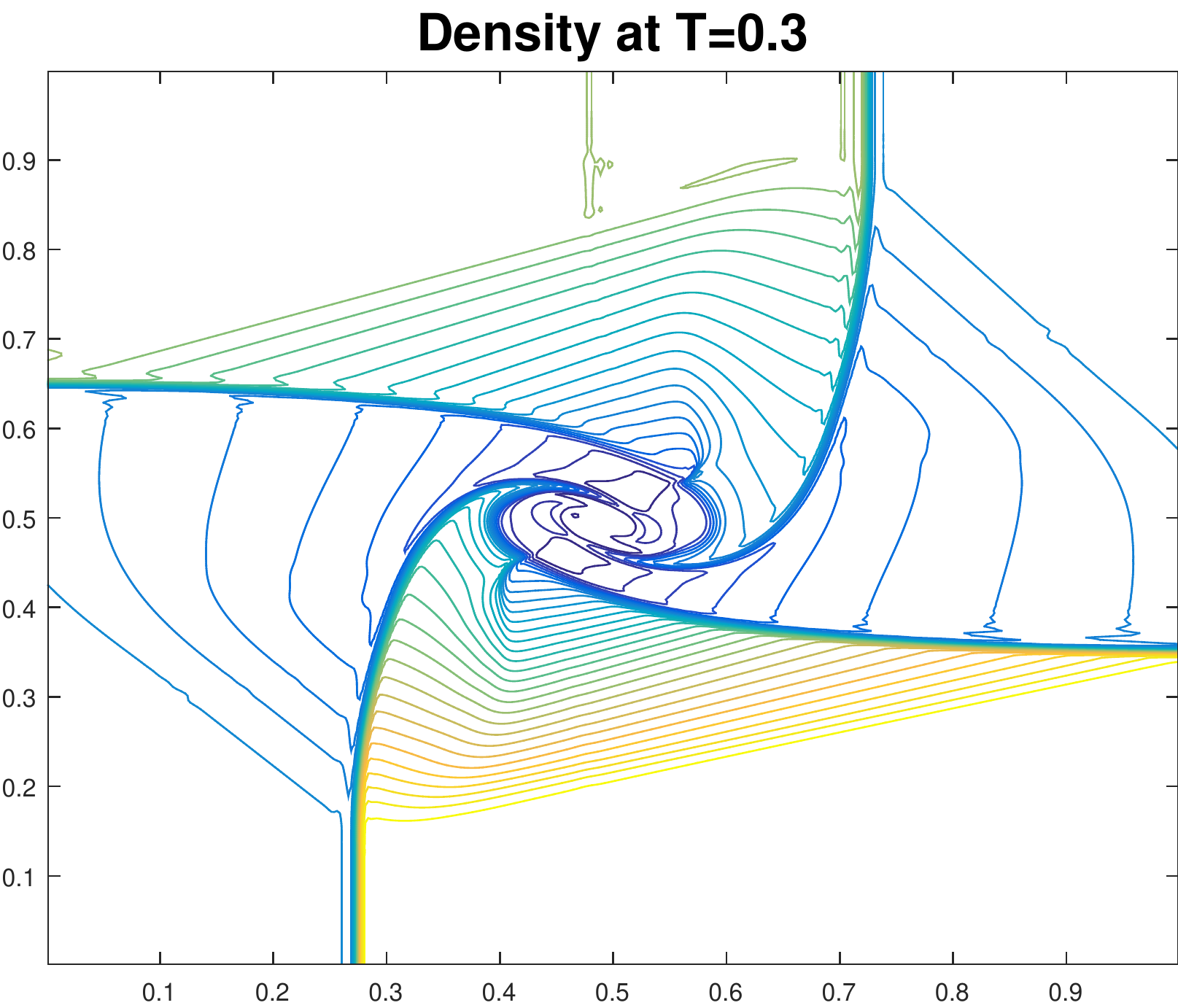}
\caption{Configuration 6. Contour plot of numerical solution of density with 30 contour levels. Left: with positive-preserving limiter; Right: with invariant-region-preserving limiter}
\label{fig:C6}
\end{figure}  

\begin{figure}[htbp]
\centering
\includegraphics[height = 1.8in]{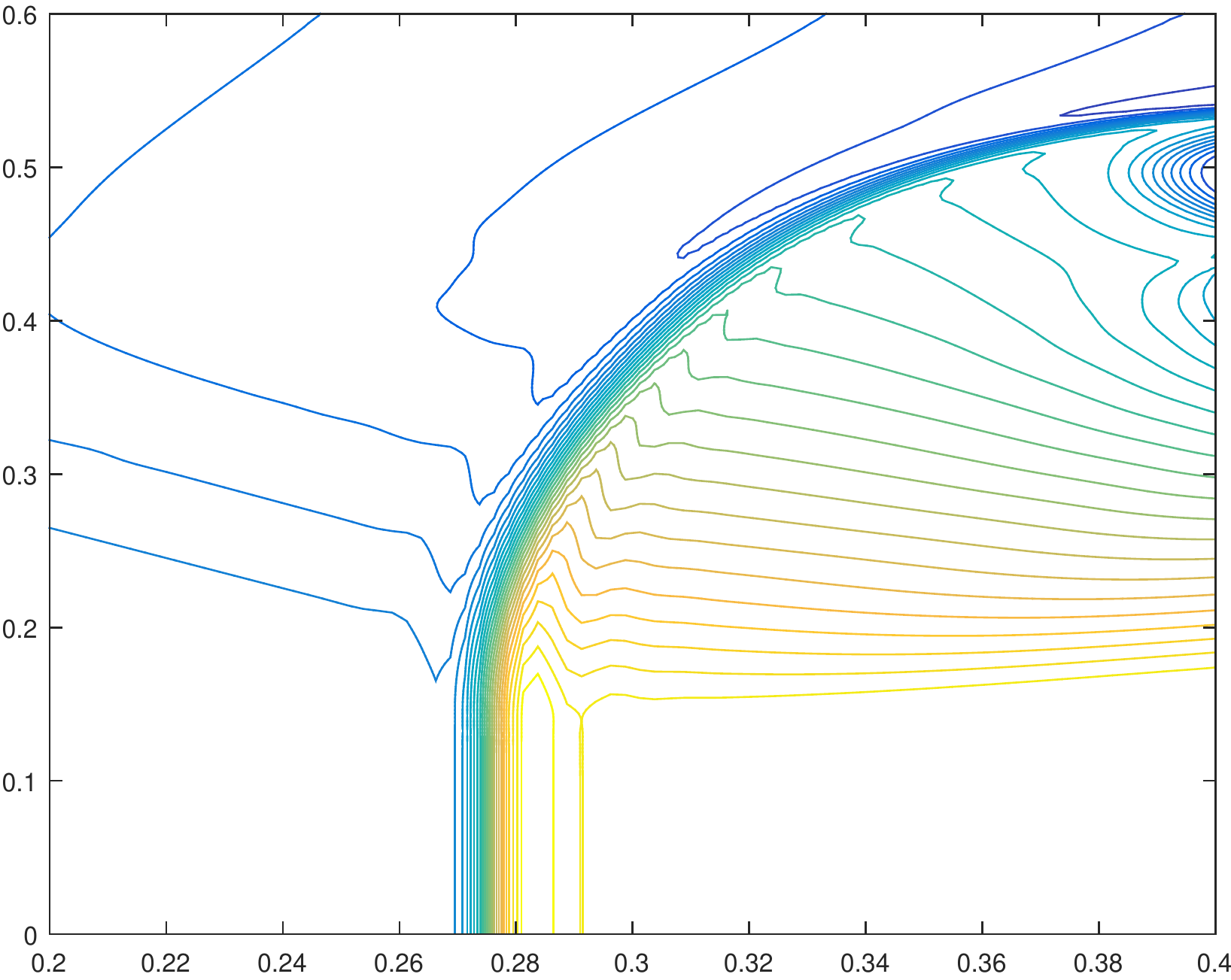}
\includegraphics[height = 1.8in]{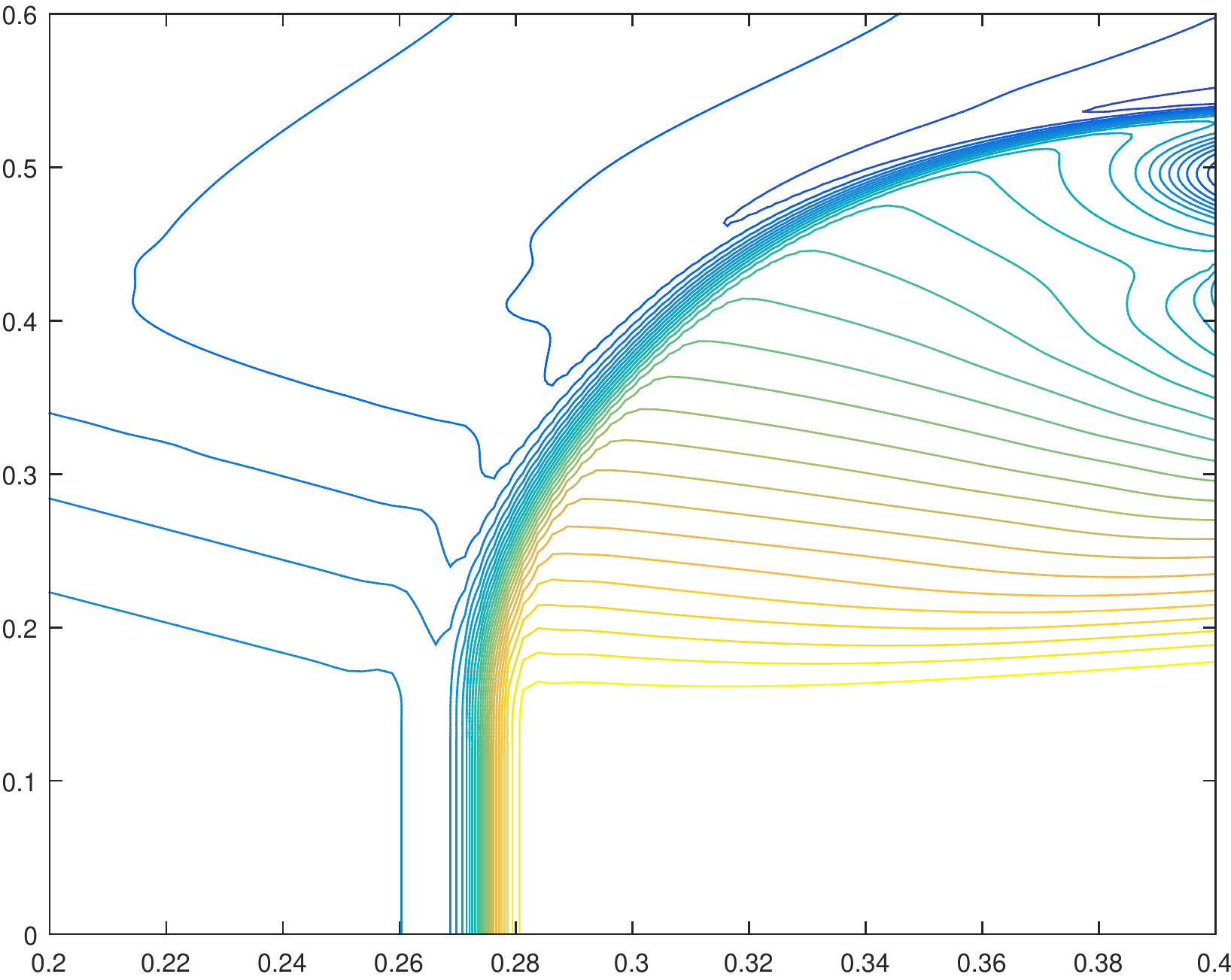}\\
\includegraphics[height = 1.8in]{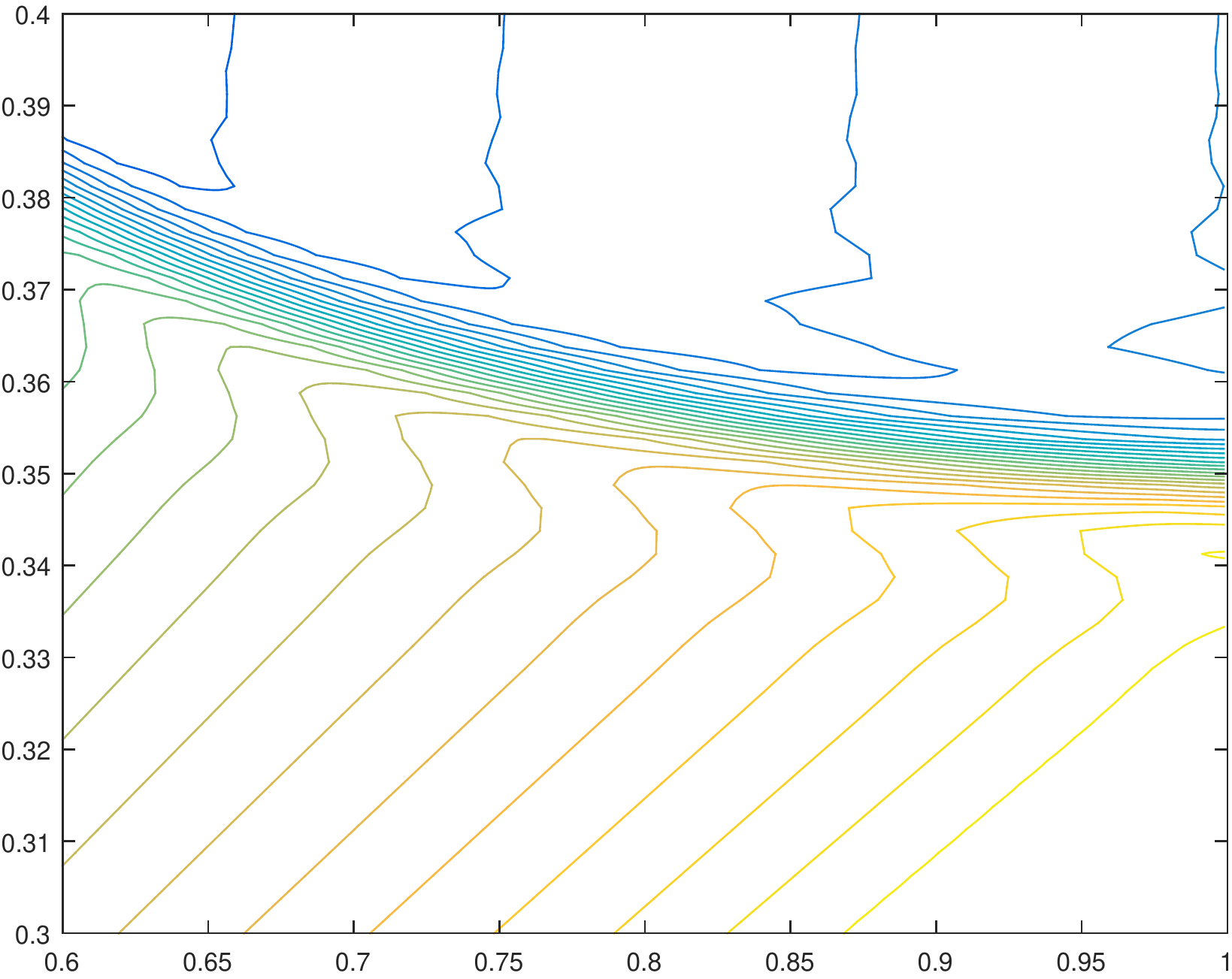}
\includegraphics[height = 1.8in]{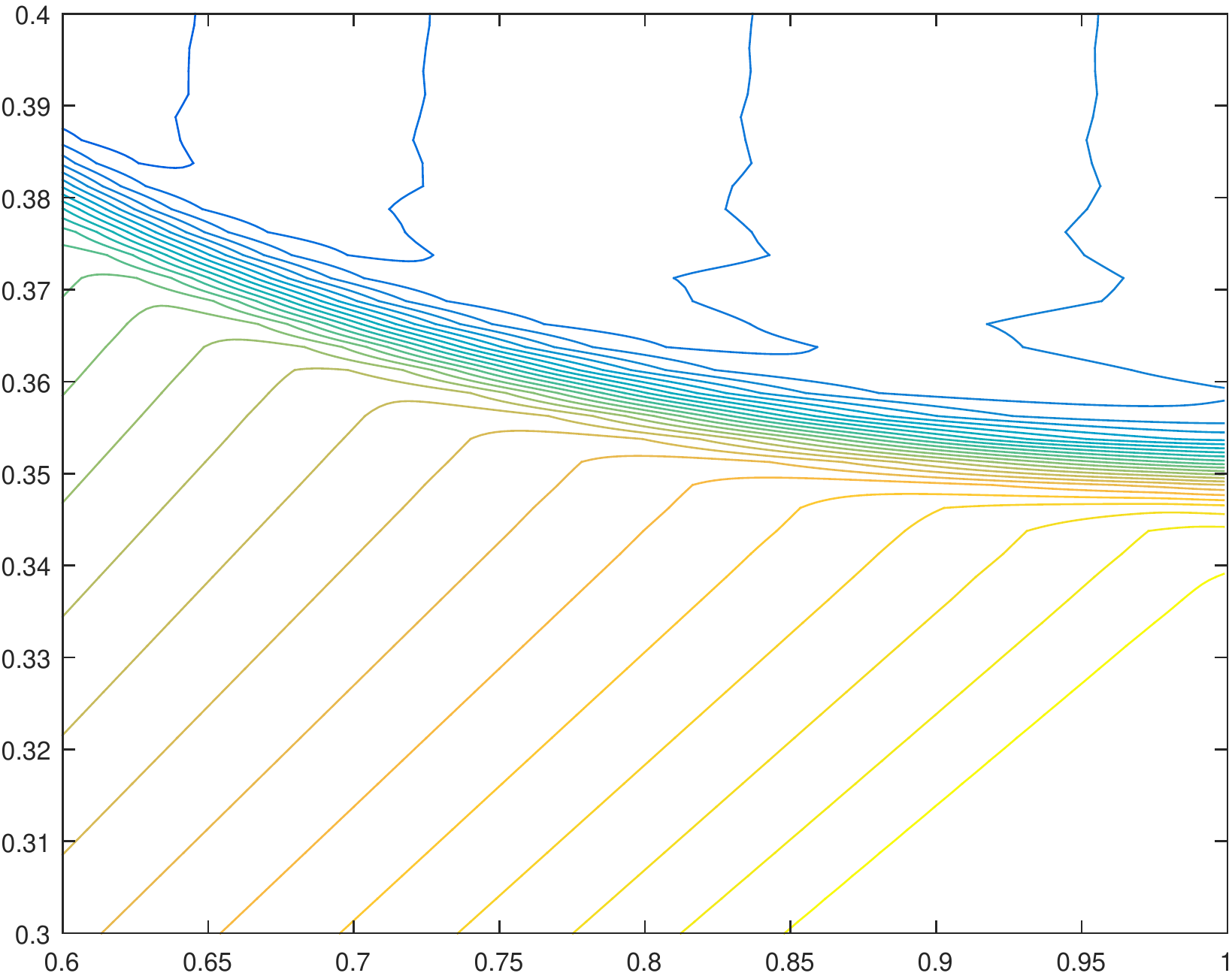}\\
\caption{Zoom-in plot of contour plots of configuration 6. Left: with positive-preserving limiter; Right: with invariant-region-preserving limiter. Top: lower-left interface; Bottom: lower-right interface.}
\label{fig:zmC6}
\end{figure}  

\begin{rem}
From the examples above, we can see that the IRP limiter presented in this work is still a mild limiter, especially for two dimensional case, and oscillations may not be completely damped even the invariant region has been preserved. For stronger oscillations, some stricter limiters may be needed. 
\end{rem}

\section{Concluding remarks}
In this paper we investigate invariant-region-preserving (IRP) DG schemes for multi-dimensional hyperbolic conservation law systems, with an application to the compressible Euler equations.  Assume that the underlying system admits a global invariant region which is a convex set in the phase space, an explicit IRP limiter is implemented in such a way that the cell averages remain in the invariant region for the entire simulation, which adds a degree of robustness to our IRP DG schemes. We  rigorously prove that the invariant region limiter maintains both conservation and high order accuracy. {The loss of accuracy might occur when the cell average is close to the boundary of the convex set. }
A generic algorithm incorporating the IRP limiter is presented for high order finite volume type schemes, and sufficient conditions are further identified if we assume the projected one-dimensional system shares the same invariant region as the full multi-dimensional hyperbolic system. We then apply the results to both one and two dimensional compressible Euler equations so to obtain high order IRP DG schemes. We demonstrate the effectiveness and efficiency of the IRP DG schemes on one- and two-dimensional compressible Euler equations. High-order accuracy is retained after 
applying the IRP limiter to a set of test problems, while some oscillations in the numerical solution are damped by the limiter as desired.

\bigskip
\section*{Acknowledgments} We would like to thank
the associate editor and two reviewers for many constructive comments that improved the presentation of the paper.

\bigskip

\appendix
\numberwithin{equation}{section}
\section{}
In this appendix, we show that for the compressible Euler equations, where the pressure function is concave but not strictly concave, Lemma \ref{U+} still holds. We consider the one-dimensional case, where the invariant region is given in (\ref{IRE}).
\begin{proof}
Since $p$ is concave, using Jensen's inequality and the assumption, we have
\begin{align*}
p(\bar{\mathbf{w}})=p\left(\frac{1}{|K|}\int _K\mathbf{w}(x)dx\right)\geq \frac{1}{|K|}\int _K p(\mathbf{w}(x))dx\geq 0.
\end{align*}
With this, we can show $p(\bar{\mathbf{w}})<0$. Otherwise if $p(\bar{\mathbf{w}})=0$, we must have $p(\mathbf{w}(x))=0$ for all $x\in K $; that is
\begin{align}\label{pconst}
p(\bar{\mathbf{w}})=p(\mathbf{w}).
\end{align}
Upon taking cell average of this relation on both sides, we have 
\begin{align*}
p(\bar{\mathbf{w}})=\frac{1}{|K|}\int _K p(\mathbf{w}(x))dx.
\end{align*}
By taking the Taylor expansion around $\bar{\mathbf{w}}$, we have
\begin{align*}
p(\mathbf{w}(x))=p(\bar{\mathbf{w}})+\triangledown _{\mathbf{w}}p(\bar{\mathbf{w}})\cdot \xi+\xi ^\top H\xi, \quad \forall x\in I, \quad \xi:=\mathbf{w}(x)-\bar{\mathbf{w}},
\end{align*}
which upon integration yields $\frac{1}{|K|}\int _K\xi ^\top H\xi dx=0$, where $H$ is the Hessian matrix of $p$:
\begin{align*}
H = (\gamma -1)\left( {\begin{array}{*{20}{c}}
{-\frac{m^2}{\rho ^3}} &{\frac{m}{\rho ^2}}  & {0}\\
{\frac{m}{\rho ^2}}&{-\frac{1}{\rho}} &0\\
0&0&0
\end{array}} \right).
\end{align*}
Since $p$ is a concave function of $\mathbf{w}=(\rho, m, E)^\top$, then $H$ is semi-definite and we have $\xi ^\top H\xi \equiv 0$. Therefore, $\xi$ must be in the eigenvalue space corresponding to the zero eigenvalue, that is $\xi=c_1v_1+c_2v_2$, where  $v_1=(0,0,1)^\top$ and $v_2=(\frac{\rho(x)}{m(x)},1,0)^\top$.  Hence $m(x)$ and $E(x)$ must be constants,  so is $\rho(x)$ following from (\ref{pconst}), which contradicts the assumption.
\end{proof}

\section{}
In this appendix, we present the proof for Lemma \ref{lemHLLC}.
\begin{proof}
When the HLLC flux is used in (\ref{1stFV}), there are sixteen different cases in total. Among them, four cases have been included in the HLL flux, so we only need to verify the other twelve cases. 

For each case, we rewrite $\mathbf{w}^{n+1}_j$ in (\ref{1stFV}) as a convex linear combination of {some terms that can be shown in the invariant region}. Here we use $\sigma _{k,l}$, $\sigma _{k,r}$ and $\sigma _{k,*}$ to denote the leftmost, rightmost and middle wave speeds at $x_k$ for $k=j\pm 1$, and $(\mathbf{w}_{*l})_k$ and $(\mathbf{w}_{*r})_k$ to denote the two intermediate states at $n-th$ time step corresponding to $x_k$ for $k=j\pm \frac{1}{2}$.\\

\underline{\textit{Case 1}}: If $\sigma _{j-\frac{1}{2},l}\geq 0$ and $\sigma _{j+\frac{1}{2},l}\leq 0\leq \sigma _{j+\frac{1}{2},*}$, then 
\begin{align*}
\hat{f}_{j-\frac{1}{2}}=f(\mathbf{w}^n_{j-1}), \quad \hat{f}_{j+\frac{1}{2}}=(f_{*l})_{j+\frac{1}{2}}=f(\mathbf{w}^n_j)+\sigma _{j+\frac{1}{2},l}\left((\mathbf{w}_{*l})_{j+\frac{1}{2}}-\mathbf{w}^n_j\right),
\end{align*}
and 
\begin{align*}
\mathbf{w}^{n+1}_j=&\mathbf{w}^n_j-\lambda ((f_{*l})_{j+\frac{1}{2}}-f(\mathbf{w}^n_{j-1}))\\
=& \left(1+\lambda\sigma_{j+\frac{1}{2},l}-\lambda\sigma _{j-\frac{1}{2},r} \right)\mathbf{w}^n_j+\left(-\lambda\sigma_{j+\frac{1}{2},l}\right)(\mathbf{w}_{*l})_{j+\frac{1}{2}}+\lambda\sigma_{j-\frac{1}{2},r}\hat{\mathbf{w}},
\end{align*}
where
\begin{align}\label{hllc1}
\hat{\mathbf{w}}=\frac{\sigma_{j-\frac{1}{2},r}\mathbf{w}^n_j-0\cdot\mathbf{w}^n_{j-1}}{\sigma_{j-\frac{1}{2}}-0}-\frac{f(\mathbf{w}^n_j)-f(\mathbf{w}^n_{j-1})}{\sigma_{j-\frac{1}{2},r}-0}.
\end{align}

\underline{\textit{Case 2}}: If $\sigma _{j-\frac{1}{2},l}\geq 0$ and $\sigma _{j+\frac{1}{2},*}\leq 0\leq \sigma_{j+\frac{1}{2},r}$, then
\begin{align*}
\hat{f}_{j-\frac{1}{2}}=f(\mathbf{w}^n_{j-1}), \quad \hat{f}_{j+\frac{1}{2}}=(f_{*r})_{j+\frac{1}{2}}=(f_{*l})_{j+\frac{1}{2}}+\sigma_{j+\frac{1}{2},*}\left((\mathbf{w}_{*r})_{j+\frac{1}{2}}-(\mathbf{w}_{*l})_{j+\frac{1}{2}} \right),
\end{align*}
and
\begin{align*}
\mathbf{w}^{n+1}_j=&\mathbf{w}^n_j-\lambda ((f_{*r})_{j+\frac{1}{2}}-f(\mathbf{w}^n_{j-1}))\\
=&\mathbf{w}^n_j-\lambda ((f_{*l})_{j+\frac{1}{2}}-f(\mathbf{w}^n_{j-1}))-\lambda\sigma_{j+\frac{1}{2},*}\left((\mathbf{w}_{*r})_{j+\frac{1}{2}}-(\mathbf{w}_{*l})_{j+\frac{1}{2}} \right)\\
=&\left(1+\lambda\sigma_{j+\frac{1}{2},l}-\lambda\sigma_{j-\frac{1}{2},r}\right)\mathbf{w}^n_j+\lambda \sigma_{j-\frac{1}{2},r}\hat{\mathbf{w}}\\
&+\lambda\left(\sigma_{j+\frac{1}{2},*}-\sigma_{j+\frac{1}{2},l}\right)(\mathbf{w}_{*l})_{j+\frac{1}{2}}+\left(-\lambda\sigma_{j+\frac{1}{2},*}\right)(\mathbf{w}_{*r})_{j+\frac{1}{2}},
\end{align*}
where $\hat{\mathbf{w}}$ is given by (\ref{hllc1}).\\

\underline{\textit{Case 3}}: If $\sigma_{j-\frac{1}{2},*}\leq 0\leq \sigma_{j-\frac{1}{2},r}$ and $\sigma_{j+\frac{1}{2},r}\leq 0$, then
\begin{align*}
\hat{f}_{j-\frac{1}{2}}=(f_{*r})_{j-\frac{1}{2}}=f(\mathbf{w}^n_j)+\sigma_{j-\frac{1}{2},r}\left((\mathbf{w}_{*r})_{j-\frac{1}{2}}-\mathbf{w}^n_j\right),\quad \hat{f}_{j+\frac{1}{2}}=f(\mathbf{w}^n_{j+1}),
\end{align*}
and
\begin{align*}
\mathbf{w}^{n+1}_j=&\mathbf{w}^n_{j}-\lambda\left(f(\mathbf{w}^n_{j+1})-(f_{*r})_{j-\frac{1}{2}}\right)\\
=&\lambda\left(\sigma_{j+\frac{1}{2},r}-\sigma_{j+\frac{1}{2},l}\right)\hat{\mathbf{w}}+\left(-\lambda\sigma_{j+\frac{1}{2},r}\right)\mathbf{w}^n_{j+1}\\
&+\left(1-\lambda\sigma_{j-\frac{1}{2},r}+\lambda\sigma_{j+\frac{1}{2},l}\right)\mathbf{w}^n_j+\lambda\sigma_{j-\frac{1}{2},r}(\mathbf{w}_{*r})_{j-\frac{1}{2}},
\end{align*}
where 
\begin{align}\label{hllc3}
\hat{\mathbf{w}}=\frac{\sigma_{j+\frac{1}{2},r}\mathbf{w}^n_{j+1}-\sigma_{j+\frac{1}{2},l}\mathbf{w}^n_j}{\sigma_{j+\frac{1}{2},r}-\sigma_{j+\frac{1}{2},l}}-\frac{f(\mathbf{w}^n_{j+1})-f(\mathbf{w}^n_j)}{\sigma_{j+\frac{1}{2},r}-\sigma_{j+\frac{1}{2},l}}.
\end{align}

\underline{\textit{Case 4}}: If $\sigma_{j-\frac{1}{2},l}\leq 0\leq \sigma_{j-\frac{1}{2},*}$ and $\sigma_{j+\frac{1}{2},r}\leq 0$, then
\begin{align*}
\hat{f}_{j-\frac{1}{2}}=(f_{*l})_{j-\frac{1}{2}}=(f_{*r})_{j-\frac{1}{2}}-\sigma_{j-\frac{1}{2},*}\left((\mathbf{w}_{*r})_{j-\frac{1}{2}}-(\mathbf{w}_{*l})_{j-\frac{1}{2}} \right), \quad \hat{f}_{j+\frac{1}{2}}=f(\mathbf{w}^n_{j+1}),
\end{align*}
and
\begin{align*}
\mathbf{w}^{n+1}_j=&\mathbf{w}^n_{j}-\lambda\left(f(\mathbf{w}_{j+1})-(f_{*l})_{j-\frac{1}{2}}\right)\\
=&\mathbf{w}^n_{j}-\lambda\left(f(\mathbf{w}^n_{j+1})-(f_{*r})_{j-\frac{1}{2}}\right)-\lambda\sigma_{j-\frac{1}{2},*}\left((\mathbf{w}_{*r})_{j-\frac{1}{2}}-(\mathbf{w}_{*l})_{j-\frac{1}{2}}\right)\\
=&\lambda\left(\sigma_{j+\frac{1}{2},r}-\sigma_{j+\frac{1}{2},l}\right)\hat{\mathbf{w}}+\left(-\lambda\sigma_{j+\frac{1}{2},r}\right)\mathbf{w}^n_{j+1}+\lambda\sigma_{j-\frac{1}{2},*}(\mathbf{w}_{*l})_{j-\frac{1}{2}}\\
&+\left(1-\lambda\sigma_{j-\frac{1}{2},r}+\lambda\sigma_{j+\frac{1}{2},l}\right)\mathbf{w}^n_j+\lambda\left(\sigma_{j-\frac{1}{2},r}-\sigma_{j-\frac{1}{2},*}\right)(\mathbf{w}_{*r})_{j-\frac{1}{2}},
\end{align*}
where $\hat{\mathbf{w}}$ is given by (\ref{hllc3}).\\

\underline{\textit{Case 5}}: If $\sigma_{j-\frac{1}{2},l}\leq 0\leq \sigma_{j-\frac{1}{2},*}$ and $\sigma_{j+\frac{1}{2},l}\leq 0\leq \sigma_{j+\frac{1}{2},*}$, then
\begin{align*}
&\hat{f}_{j-\frac{1}{2}}=(f_{*l})_{j-\frac{1}{2}}=(f_{*r})_{j-\frac{1}{2}}-\sigma_{j-\frac{1}{2},*}\left((\mathbf{w}_{*r})_{j-\frac{1}{2}}-(\mathbf{w}_{*l})_{j-\frac{1}{2}} \right),\\
&\hat{f}_{j+\frac{1}{2}}=(f_{*l})_{j+\frac{1}{2}}=f(\mathbf{w}^n_j)+\sigma _{j+\frac{1}{2},l}\left((\mathbf{w}_{*l})_{j+\frac{1}{2}}-\mathbf{w}^n_j\right),
\end{align*}
and
\begin{align*}
\mathbf{w}^{n+1}_j=&\mathbf{w}^n_j-\lambda\left((f_{*l})_{j+\frac{1}{2}}-(f_{*l})_{j-\frac{1}{2}} \right)\\
=&\left(1+\sigma_{j+\frac{1}{2},l}-\sigma_{j-\frac{1}{2},r}\right)\mathbf{w}^n_j+\left(-\lambda\sigma_{j+\frac{1}{2},l}\right)(\mathbf{w}_{*l})_{j+\frac{1}{2}}\\
&+\lambda\left(\sigma_{j-\frac{1}{2},r}-\sigma_{j-\frac{1}{2},*}\right)(\mathbf{w}_{*r})_{j-\frac{1}{2}}+\lambda\sigma_{j-\frac{1}{2},*}(\mathbf{w}_{*l})_{j-\frac{1}{2}}.
\end{align*}

\underline{\textit{Case 6}}: If $\sigma_{j-\frac{1}{2},*}\leq 0\leq \sigma_{j-\frac{1}{2},r}$ and $\sigma_{j+\frac{1}{2},*}\leq 0\leq \sigma_{j+\frac{1}{2},r}$, then
\begin{align*}
&\hat{f}_{j-\frac{1}{2}}=(f_{*r})_{j-\frac{1}{2}}=f(\mathbf{w}^n_j)+\sigma_{j-\frac{1}{2},r}\left((\mathbf{w}_{*r})_{j-\frac{1}{2}}-\mathbf{w}^n_j \right),\\
&\hat{f}_{j+\frac{1}{2}}=(f_{*r})_{j+\frac{1}{2}}=(f_{*l})_{j+\frac{1}{2}}+\sigma_{j+\frac{1}{2},*}\left((\mathbf{w}_{*r})_{j+\frac{1}{2}}-(\mathbf{w}_{*l})_{j+\frac{1}{2}}\right),
\end{align*}
then
\begin{align*}
\mathbf{w}^{n+1}_j=&\mathbf{w}^n_j-\lambda\left((f_{*r})_{j+\frac{1}{2}}-(f_{*r})_{j-\frac{1}{2}} \right)\\
=&\left(1+\lambda\sigma_{j+\frac{1}{2},l}-\lambda\sigma_{j-\frac{1}{2},r}\right)\mathbf{w}^n_j+\left(-\sigma_{j+\frac{1}{2},*}\right)(\mathbf{w}_{*r})_{j+\frac{1}{2}}\\
&+\lambda\left(\sigma_{j+\frac{1}{2},*}-\sigma_{j+\frac{1}{2},l}\right)(\mathbf{w}_{*l})_{j+\frac{1}{2}}+\lambda\sigma_{j-\frac{1}{2},r}(\mathbf{w}_{*r})_{j-\frac{1}{2}}.
\end{align*}

\underline{\textit{Case 7}}: If $\sigma_{j-\frac{1}{2},*}\leq 0\leq \sigma_{j-\frac{1}{2},r}$ and $\sigma_{j+\frac{1}{2},l}\leq 0\leq \sigma_{j+\frac{1}{2},*}$, then 
\begin{align*}
&\hat{f}_{j-\frac{1}{2}}=(f_{*r})_{j-\frac{1}{2}}=f(\mathbf{w}^n_j)+\sigma_{j-\frac{1}{2},r}\left((\mathbf{w}_{*r})_{j-\frac{1}{2}}-\mathbf{w}^n_j \right),\\
&\hat{f}_{j+\frac{1}{2}}=(f_{*l})_{j+\frac{1}{2}}=f(\mathbf{w}^n_j)+\sigma _{j+\frac{1}{2},l}\left((\mathbf{w}_{*l})_{j+\frac{1}{2}}-\mathbf{w}^n_j\right),
\end{align*}
and
\begin{align*}
\mathbf{w}^{n+1}_j=&\mathbf{w}^n_j-\lambda\left((f_{*r})_{j+\frac{1}{2}}-(f_{*l})_{j-\frac{1}{2}} \right)\\
=&\left(1+\lambda\sigma_{j+\frac{1}{2},l}-\lambda\sigma_{j-\frac{1}{2},r} \right)\mathbf{w}^n_j+\left(-\lambda\sigma_{j+\frac{1}{2},l}\right)(\mathbf{w}_{*l})_{j+\frac{1}{2}}+\lambda\sigma_{j-\frac{1}{2},r}(\mathbf{w}_{*r})_{j-\frac{1}{2}}.
\end{align*}

\underline{\textit{Case 8}}: If $\sigma_{j-\frac{1}{2},l}\leq 0\leq \sigma_{j-\frac{1}{2},*}$ and $\sigma_{j+\frac{1}{2},*}\leq 0\leq \sigma_{j+\frac{1}{2},r}$, then 
\begin{align*}
&\hat{f}_{j-\frac{1}{2}}=(f_{*l})_{j-\frac{1}{2}}=(f_{*r})_{j-\frac{1}{2}}-\sigma_{j-\frac{1}{2},*}\left((\mathbf{w}_{*r})_{j-\frac{1}{2}}-(\mathbf{w}_{*l})_{j-\frac{1}{2}} \right), \\
&\hat{f}_{j+\frac{1}{2}}=(f_{*r})_{j+\frac{1}{2}}=(f_{*l})_{j+\frac{1}{2}}+\sigma_{j+\frac{1}{2},*}\left((\mathbf{w}_{*r})_{j+\frac{1}{2}}-(\mathbf{w}_{*l})_{j+\frac{1}{2}}\right),
\end{align*}
and
\begin{align*}
\mathbf{w}^{n+1}_j=&\mathbf{w}^n_j-\lambda\left((f_{*r})_{j+\frac{1}{2}}-(f_{*l})_{j-\frac{1}{2}} \right)\\
=&\left(1+\lambda\sigma_{j+\frac{1}{2},l}-\lambda\sigma_{j-\frac{1}{2},r}\right)\mathbf{w}^n_j+\lambda\sigma_{j-\frac{1}{2},*}(\mathbf{w}_{*l})_{j-\frac{1}{2}}+\left(-\lambda\sigma_{j+\frac{1}{2},*} \right)(\mathbf{w}_{*r})_{j+\frac{1}{2}}\\
&+\left(\lambda\sigma_{j+\frac{1}{2},*}-\lambda\sigma_{j+\frac{1}{2},l} \right)(\mathbf{w}_{*l})_{j+\frac{1}{2}}+\left(\lambda\sigma_{j-\frac{1}{2},r}-\lambda\sigma_{j-\frac{1}{2},*} \right)(\mathbf{w}_{*r})_{j-\frac{1}{2}}.
\end{align*}

\underline{\textit{Case 9}}: If $\sigma_{j-\frac{1}{2},*}\leq 0\leq \sigma_{j-\frac{1}{2},r}$ and $0\leq \sigma_{j+\frac{1}{2},l}$, then
\begin{align*}
\hat{f}_{j-\frac{1}{2}}=(f_{*r})_{j-\frac{1}{2}}=f(\mathbf{w}^n_j)+\sigma_{j-\frac{1}{2},r}\left((\mathbf{w}_{*r})_{j-\frac{1}{2}}-\mathbf{w}^n_j \right), \quad \hat{f}_{j+\frac{1}{2}}=f(\mathbf{w}^n_j),
\end{align*}
and
\begin{align*}
\mathbf{w}^{n+1}_j=&\mathbf{w}^n_j-\lambda\left(f(\mathbf{w}^n_j)- (f_{*r})_{j-\frac{1}{2}}\right)\\
=&\left(1-\lambda\sigma_{j-\frac{1}{2},r}\right)\mathbf{w}^n_j+\lambda\sigma_{j-\frac{1}{2},r}(\mathbf{w}_{*r})_{j-\frac{1}{2}}.
\end{align*}

\underline{\textit{Case 10}}: If $\sigma_{j-\frac{1}{2},l}\leq 0\leq \sigma_{j-\frac{1}{2},*}$ and $0\leq \sigma_{j+\frac{1}{2},l}$, then
\begin{align*}
\hat{f}_{j-\frac{1}{2}}=(f_{*l})_{j-\frac{1}{2}}=(f_{*r})_{j-\frac{1}{2}}-\sigma_{j-\frac{1}{2},*}\left((\mathbf{w}_{*r})_{j-\frac{1}{2}}-(\mathbf{w}_{*l})_{j-\frac{1}{2}} \right), \quad \hat{f}_{j+\frac{1}{2}}=f(\mathbf{w}^n_j),
\end{align*}
and
\begin{align*}
\mathbf{w}^{n+1}_j=&\mathbf{w}^n_j-\lambda\left(f(\mathbf{w}^n_j)- (f_{*l})_{j-\frac{1}{2}}\right)\\
=&\left(1-\lambda\sigma_{j-\frac{1}{2},r} \right)\mathbf{w}^n_j+\left(\lambda_{j-\frac{1}{2},r}-\lambda\sigma_{j-\frac{1}{2},*}\right)(\mathbf{w}_{*r})_{j-\frac{1}{2}}+\lambda\sigma_{j-\frac{1}{2},*}(\mathbf{w}_{*l})_{j-\frac{1}{2}}.
\end{align*}

\underline{\textit{Case 11}}: If $\sigma_{j-\frac{1}{2},r}\leq 0$ and $\sigma_{j+\frac{1}{2},l}\leq 0\leq \sigma_{j+\frac{1}{2},*}$, then
\begin{align*}
\hat{f}_{j-\frac{1}{2}}=f(\mathbf{w}^n_j), \quad \hat{f}_{j+\frac{1}{2}}=(f_{*l})_{j+\frac{1}{2}}=f(\mathbf{w}^n_j)+\sigma_{j+\frac{1}{2},l}\left((\mathbf{w}_{*l})_{j+\frac{1}{2}}-\mathbf{w}^n_j\right),
\end{align*}
and
\begin{align*}
\mathbf{w}^{n+1}_j=&\mathbf{w}^n_j-\lambda\left((f_{*l})_{j+\frac{1}{2}}-f(\mathbf{w}^n_j)\right)\\
=&\left(1-\lambda\sigma_{j+\frac{1}{2},l}\right)\mathbf{w}^n_j+\left(-\lambda\sigma_{j+\frac{1}{2},l}\right)(\mathbf{w}_{*l})_{j+\frac{1}{2}}.
\end{align*}

\underline{\textit{Case 12}}: If $\sigma_{j-\frac{1}{2},r}\leq 0$ and $\sigma_{j+\frac{1}{2},*}\leq 0\leq \sigma_{j+\frac{1}{2},r}$, then
\begin{align*}
\hat{f}_{j-\frac{1}{2}}=f(\mathbf{w}^n_j), \quad \hat{f}_{j+\frac{1}{2}}=(f_{*r})_{j+\frac{1}{2}}=(f_{*l})_{j+\frac{1}{2}}+\sigma_{j+\frac{1}{2},*}\left((\mathbf{w}_{*r})_{j+\frac{1}{2}}-(\mathbf{w}_{*l})_{j+\frac{1}{2}}\right),
\end{align*}
and
\begin{align*}
\mathbf{w}^{n+1}_j=&\mathbf{w}^n_j-\lambda\left((f_{*r})_{j+\frac{1}{2}}-f(\mathbf{w}^n_j)\right)\\
=&\left(1-\lambda\sigma_{j+\frac{1}{2},l}\right)\mathbf{w}^n_j+\left(\lambda\sigma_{j+\frac{1}{2},*}-\lambda\sigma_{j+\frac{1}{2},l}\right)(\mathbf{w}_{*l})_{j+\frac{1}{2}}+\left(-\lambda\sigma_{j+\frac{1}{2},*}\right)(\mathbf{w}_{*r})_{j+\frac{1}{2}}.
\end{align*}
For $\mathbf{w}^n_j, \mathbf{w}^n_{j\pm 1}\in \Sigma$, the intermediate states $(\mathbf{w}_{*r})_k$ and $(\mathbf{w}_{*l})_k$, $k=j\pm \frac{1}{2}$ are in $\Sigma$. Also notice that  $\hat{\mathbf{w}}$ in Case 1-4 are all in the form of (\ref{id}), the cell average of some exact Riemann solutions, and therefore lie in $\Sigma _0$ by Lemma \ref{U+}. Hence $\lambda\sigma\leq \frac{1}{2}$ is a sufficient condition for $\mathbf{w}^{n+1}_j$ defined in (\ref{1stFV}) to be in $\Sigma _0$. 

\end{proof}

\bibliographystyle{plain}

\begin{thebibliography}{10}

\bibitem{BSL}
F. Bouchut, S. Jin, and X. Li. 
\newblock Numerical approximations of pressureless and isothermal gas dynamics.
\newblock{\em SIAM Journal on Numerical Analysis} 41(1): 135-158, 2003.

\bibitem{CCS77}
K. N. Chueh, C. C. Conley, and J. A. Smoller. 
\newblock Positively invariant regions for systems of nonlinear diffusion equations. 
\newblock {\em Indiana Univ. Math. J.}, 26(2):373--392,1977.

\bibitem{CLS89}
B. Cockburn, S.Y. Lin, and C.-W. Shu.
\newblock TVB Runge-Kutta local projection discontinuous Galerkin finite element method for conservation laws. {III}. One dimensional
  systems.
\newblock {\em Journal of Computational Physics}, 84(1):90--113, 1989.

\bibitem{CS89}
B. Cockburn and C.-W. Shu.
\newblock TVB Runge-Kutta local projection discontinuous Galerkin finite element method for conservation laws. {II}. General framework.
\newblock {\em Mathematics of computation}, 52(186):411--435, 1989.

\bibitem{CS98}
B. Cockburn and C.-W. Shu. 
\newblock The Runge-Kutta discontinuous Galerkin method for conservation laws V: multidimensional systems.
\newblock{\em Journal of Computational Physics}, 141: 199--224, 1998.

\bibitem{Colella}
P. Colella.
\newblock Multidimensional upwind methods for hyperbolic conservation laws. 
\newblock{\em Journal of Computational Physics}, 87(1):171--200, 1990.

%

\bibitem{Fr95}
H. Frid
\newblock Invariant regions under Lax-Friedrichs scheme for multidimensional systems of conservation laws
\newblock {\em Discrete and Continuous Dynamical Systems}, 1(4): 585--593, 1995.

\bibitem{Fr01}
H. Frid.
\newblock  Maps of convex sets and invariant regions for finite-difference systems of conservation laws.
\newblock {\em Archive for rational mechanics and analysis}, 160(3): 245--269, 2001.


\bibitem{GP16}
J. L. Guermond and  B. Popov.
\newblock Invariant domains and first-order continuous finite element approximation for hyperbolic systems
\newblock {\em SIAM Journal on Numerical Analysis},  54(4): 2466--2489, 2016.


\bibitem{GP17}
J. L. Guermond, M. Nazarov, B. Popov and I. Tomas
\newblock Second-order invariant domain preserving approximation of the Euler equations using convex limiting.
\newblock {\em arXiv preprint} arXiv:1710.00417, 2017.

\bibitem{H83}
A. Harten.
\newblock On the symmetric form of system of conservation laws with entropy.
\newblock {\em J. Comput. Phys.}, 49: 151-164, 1983.

\bibitem{HLL}
A. Harten, P. D. Lax and B. Van Leer. 
\newblock On upstream differencing and Godunov-type schemes for hyperbolic conservation laws. 
\newblock{\em SIAM review}, 25(1): 35-61, 1983.

\bibitem{Hoff85}
D. Hoff. 
\newblock Invariant regions for systems of conservation laws.
\newblock {\em Transactions of the American Mathematical Society}, 289 (2):591-610, 1985.

\bibitem{JLPsy}
Y. Jiang and H. Liu.
\newblock An invariant-region-preserving limiter for DG schemes to isentropic Euler equations.
\newblock{\em Numer. Methods Partial Differ. Equ. (2018), accepted for publication.}

\bibitem{Hyp16}
Y. Jiang and H. Liu.
\newblock An invariant-region-preserving limiter for the DG method to compressible Euler equations.
\newblock{\em Proceedings of the XVI International Conference on Hyperbolic Problems: Theory, Numeric and Applications held in Aachen, August 1-5, 2016. }

\bibitem{KP94} 
B. Khobalatte and B. Perthame.
\newblock Maximum principle on the entropy and second-order kinetic schemes.
\newblock {\em Mathematics of Computation}, 62: 119--131, 1994.


\bibitem{La71} 
 P.D. Lax. 
\newblock  Shock waves and entropy.
\newblock  Contributions to Nonlinear Functional Analysis (E. Zarantonello, ed.), Academic Press, New York, 1971.

\bibitem{La73} 
 P.D. Lax. 
\newblock  Hyperbolic systems of conservation laws and the mathematical theory of shock waves.
\newblock  CBMS-NSF 11, 1973.

\bibitem{LL}
P.D. Lax and X.D. Liu. 
\newblock Solution of two-dimensional Riemann problems of gas dynamics by positive schemes
\newblock{\em SIAM Journal on Scientific Computing} 19(2): 319-340, 1998.

\bibitem{CLLeVeque}
Randall J. LeVeque. 
\newblock Numerical methods for conservation laws (Vol. 132).
\newblock{\em Basel: Birkhäuser}, 1992.

\bibitem{FVLeVeque}
Randall J. LeVeque. 
\newblock Finite volume methods for hyperbolic problems (Vol. 31). 
\newblock{\em Cambridge university press}, 2002.



\bibitem{PQ94}
B. Perthame and Y. Qiu. 
\newblock A variant of Van LeerÕs method for multidimensional
systems of conservation laws. 
\newblock {\em J. Comput. Phys.}, 112(2):370--381,1994.

\bibitem{PS96}
B. Perthame and C.-W. Shu.
\newblock On positivity preserving finite volume schemes for Euler equations.
\newblock {\em Numerische Mathematik}, 73: 119--130,  1996.

\bibitem{Rohde}
A. Rohde 
\newblock Eigenvalues and eigenvectors of the Euler equations in general geometries
\newblock{\em AIAA paper}, 2609, 2001.


\bibitem{SO88}
C.-W. Shu and S. Osher.
\newblock Efficient implementation of essentially non-oscillatory shock-capturing schemes.
\newblock{\em Journal of Computational Physics}, 77: 439--471, 1988.

\bibitem{Sm94}
J. Smoller.
\newblock Shock waves and reaction-diffusion equations.
\newblock{\em Springer Science \& Business Media}, 198--212, 1994.

\bibitem{Ta86}
E. Tadmor.
\newblock A minimum entropy principle in the gas dynamics equations.
\newblock {\em Applied Numerical Mathematics}, 2:211--219, 1986.

\bibitem{TSS94}
E. Toro, M. Spruce and W. Speares. 
\newblock  Restoration of the contact surface in the HLL--Riemann solver. 
\newblock {\em Shock Waves},  4:25--34, 1994. 

\bibitem{YWS}
Y. Yang, D.M. Wei and C.W. Shu.
\newblock Discontinuous Galerkin method for Krause's consensus models and pressureless Euler equations.
\newblock{\em Journal of Computational Physics}, 252:109-127, 2013.


\bibitem{ZCY}
T. Zhang, G.-Q. Chen and S. Yang.
\newblock On the 2-D Riemann Problem for the Compressible Euler Equations, I. Interaction of Shocks and Rarefaction Waves.
\newblock{\em Discrete and Continuous Dynamical Systems} 1: 555-584, 1995.

\bibitem{ZS10a}
X. Zhang and C.-W. Shu.
\newblock On maximum-principle-satisfying high order schemes for scalar conservation laws.
\newblock{\em Journal of Computational Physics}, 229: 3091--3120, 2010.

\bibitem{ZS10b}
X. Zhang and C.-W. Shu.
\newblock On positivity preserving high order discontinuous Galerkin schemes for compressible Euler equations on rectangular meshes.
 \newblock{\em Journal of Computational Physics}, 229: 8918--8934, 2010.
 
\bibitem{ZS12}
X. Zhang and C.-W. Shu.
\newblock  A minimum entropy principle of high order schemes for gas dynamics equations.
\newblock {\em Numerische Mathematik}, 121:545-563, 2012

\bibitem{ZXS12}
X. Zhang, Y. Xia and C.W. Shu.
\newblock Maximum-principle-satisfying and positivity-preserving high order discontinuous Galerkin schemes for conservation laws on triangular meshes.
\newblock{\em Journal of Scientific Computing}, 50(1): 29-62, 2012.


\end{thebibliography}

\end{document}